\documentclass[12pt]{amsart}
\usepackage[all]{xy}

\newcommand{\Cdb}{\mbox{$\mathbb{C}$}}
\newcommand{\Ddb}{\mbox{$\mathbb{D}$}}
\newcommand{\Edb}{\mbox{$\mathbb{E}$}}

\newcommand{\Pdb}{\mbox{$\mathbb{P}$}}

\newcommand{\Rdb}{\mbox{$\mathbb{R}$}}

\newcommand{\Tdb}{\mbox{$\mathbb{T}$}}

\newcommand{\Zdb}{\mbox{$\mathbb{Z}$}}

\newcommand{\A}{\mbox{${\mathcal A}$}}

\newcommand{\C}{\mbox{${\mathcal C}$}}

\newcommand{\E}{\mbox{${\mathcal E}$}}
\newcommand{\F}{\mbox{${\mathcal F}$}}

\renewcommand{\H}{\mbox{${\mathcal H}$}}

\newcommand{\K}{\mbox{${\mathcal K}$}}

\renewcommand{\P}{\mbox{${\mathcal P}$}}

\renewcommand{\S}{\mbox{${\mathcal S}$}}

\newcommand{\V}{\mbox{${\mathcal V}$}}

\newcommand{\norm}[1]{\Vert#1\Vert}
\newcommand{\bignorm}[1]{\bigl\Vert#1\bigr\Vert}
\newcommand{\Bignorm}[1]{\Bigl\Vert#1\Bigr\Vert}

\newcommand{\cbnorm}[1]{\Vert#1\Vert_{cb}}
\newcommand{\bigcbnorm}[1]{\bigl\Vert#1\bigr\Vert_{cb}}
\newcommand{\Bigcbnorm}[1]{\Bigl\Vert#1\Bigr\Vert_{cb}}

\newcommand{\nnorm}[1]{\Vert#1\Vert_{n}}

\newcommand{\minten}{\otimes_{\rm min}}

\newtheorem{theorem}{Theorem}[section]
\newtheorem{lemma}[theorem]{Lemma}
\newtheorem{corollary}[theorem]{Corollary}
\newtheorem{proposition}[theorem]{Proposition}
\newtheorem{definition}[theorem]{Definition}
\theoremstyle{remark}
\newtheorem{remark}[theorem]{\bf Remark}
\theoremstyle{definition}

\numberwithin{equation}{section}

\begin{document}

\title[]{Completely $1$-complemented subspaces of Schatten spaces}

\author{Christian Le Merdy, \'Eric Ricard and Jean Roydor}
\address{Laboratoire de Math\'ematiques\\ Universit\'e de  Franche-Comt\'e
\\ 25030 Besan\c con Cedex\\ France}
\email{clemerdy@univ-fcomte.fr}
\address{Laboratoire de Math\'ematiques\\ Universit\'e de  Franche-Comt\'e
\\ 25030 Besan\c con Cedex\\ France}
\email{eric.ricard@univ-fcomte.fr}
\address{Laboratoire de Math\'ematiques\\ Universit\'e de  Franche-Comt\'e
\\ 25030 Besan\c con Cedex\\ France}
\email{jean.roydor@univ-fcomte.fr}
\date{\today}

\begin{abstract} We consider the Schatten spaces $S^p$ in the framework of
operator space theory and for any $1\leq p\not=2<\infty$, we
characterize the completely $1$-complemented subspaces of $S^p$.
They turn out to be the direct sums of spaces of the form
$S^p(H,K)$, where $H,K$ are Hilbert spaces. This result is related
to some previous work of Arazy-Friedman giving a description of
all $1$-complemented subspaces of $S^p$ in terms of the Cartan
factors of types 1-4. We use operator space structures on these
Cartan factors regarded as subspaces of appropriate noncommutative
$L^p$-spaces. Also we show that for any $n\geq 2$, there is a
triple isomorphism on some Cartan factor of type 4 and of
dimension $2n$ which is not completely isometric, and we
investigate $L^p$-versions of such isomorphisms.
\end{abstract}

\maketitle

\bigskip\noindent
{\it 2000 Mathematics Subject Classification : 46L07, 46L89,
17C65}

\bigskip

\section{Introduction}
Let $\H,\K$ be Hilbert spaces. For any $p\geq 1$, let $S^p(\H,\K)$
be the Schatten space of all operators $x\colon\H\to \K$ such that
$\norm{x}_p=\bigl(tr(\vert x\vert^p)\bigr)^{\frac{1}{p}}$ is
finite. Let $X\subset S^p(\H,\K)$ be a (closed) subspace. We say
that $X$ is $1$-complemented in $S^p(\H,\K)$ if it is the range of
a contractive projection $P\colon S^p(\H,\K)\to S^p(\H,\K)$. In
their remarkable memoirs \cite{AF2, AF3}, Arazy and Friedman gave
a complete classification of all such subspaces (for $p\not=2$),
in terms of Cartan factors of types 1-4.

In this paper we consider Schatten spaces and their complemented
subspaces in the framework of operator spaces and completely
bounded maps. Following Pisier's work \cite{P1}, we regard
$S^p(\H,\K)$ as an operator space and we give a complete
description of the completely $1$-complemented subspaces of
$S^p(\H,\K)$, that is, spaces $X\subset S^p(\H,\K)$ which are the
range of a completely contractive projection of $S^p(\H,\K)$.

The statement of our main result, Theorem \ref{1Main} below,
requires some tensor product definitions and some notation. For
any Hilbert spaces $H,H',K,K'$, we will consider the natural
embedding
$$
S^p(H',K')\otimes S^p(H,K)\,\subset\,
S^p(H'\mathop{\otimes}\limits^2 H,K'\mathop{\otimes}\limits^2 K),
$$
where $\mathop{\otimes}\limits^2$ denotes the Hilbertian tensor
product. Thus for any subspace $Z\subset S^p(H',K')$ and any $a\in
S^p(H,K)$, we will regard
$$
Z\otimes a :\, =\{z\otimes a\, :\, z\in Z\}
$$
as a subspace of $S^p(H'\mathop{\otimes}\limits^2
H,K'\mathop{\otimes}\limits^2 K)$.

If $I,J$ are two index sets, we set
$S^{p}_{I,J}=S^p(\ell^2_J,\ell^2_I)$ and we write $S^p_I =
S^p_{I,I}$.  With this notation, $S^p_{I,J}\otimes S^p(H,K)\subset
S^p(\ell^2_J(H),\ell^2_I(K))$, where
$\ell^2_J(H)=\ell^2_J\mathop{\otimes}\limits^2 H$ is the
$2$-direct sum of $J$ copies of $H$.

Next  we recall that if $(H_\alpha)_\alpha$ and
$(K_\alpha)_\alpha$ are two families of Hilbert spaces, then we
have a natural isometric embedding
$$
\mathop{\oplus}\limits_{\alpha}^{p} S^p(H_\alpha,K_\alpha)\,
\subset\, S^p\bigl(\mathop{\oplus}\limits_{\alpha}^{2}
H_\alpha,\mathop{\oplus}\limits_{\alpha}^{2} K_\alpha\bigr),
$$
where $\mathop{\oplus}\limits_{\alpha}^{p} S^p(H_\alpha,K_\alpha)$
denotes the $p$-direct sum of the $S^p(H_\alpha,K_\alpha)$'s and
$\mathop{\oplus}\limits_{\alpha}^{2} H_\alpha$ denotes the
$2$-direct sum of the $H_\alpha$'s. This is obtained by
identifying any $(x_\alpha)_\alpha$ in
$\mathop{\oplus}\limits_{\alpha}^{p} S^p(H_\alpha,K_\alpha)$ with
the `diagonal' operator $\mathop{\oplus}\limits_{\alpha}^{2}
H_\alpha\to \mathop{\oplus}\limits_{\alpha}^{2} K_\alpha$ taking
any $(\xi_\alpha)_\alpha$ to $(x_\alpha(\xi_\alpha))_\alpha$.

\begin{theorem}\label{1Main} Let $\H,\K$ be Hilbert spaces, let
$1\leq p\not= 2<\infty$ and let $X\subset S^p(\H,\K)$ be a
subspace. The following are equivalent.
\begin{itemize}
\item [(i)] $X$ is completely $1$-complemented in $S^p(\H,\K)$.
\item [(ii)] $X$ is $[2]$-$1$-complemented in $S^p(\H,\K)$. \item
[(iii)] There exist, for some set $A$, two families of indices
$(I_\alpha)_{\alpha\in A}$ and $(J_\alpha)_{\alpha\in A}$, a
family  $(H_\alpha)_{\alpha\in A}$ of Hilbert spaces, as well as
operators $a_\alpha\in S^p(H_\alpha)$, and two linear isometries
$$
U\colon \mathop{\oplus}\limits_{\alpha\in A}^2
\ell^2_{J_\alpha}(H_\alpha)\longrightarrow
\H\qquad\hbox{and}\qquad V\colon \mathop{\oplus}\limits_{\alpha\in
A}^2 \ell^2_{I_\alpha}(H_\alpha)\longrightarrow \K
$$
such that
$$
X \,=\, V\Bigl(\mathop{\bigoplus}\limits_{\alpha}^p
S^p_{I_\alpha,J_\alpha}\otimes a_\alpha \Bigr)U^*.
$$
\item [(iv)] There exist, for some set $A$, two families of
indices $(I_\alpha)_{\alpha\in A}$ and $(J_\alpha)_{\alpha\in A}$
such that $X$ is completely isometric to the $p$-direct sum
$\mathop{\oplus}\limits_{\alpha}^p S^p_{I_\alpha,J_\alpha}$.
\end{itemize}
\end{theorem}

See Definition \ref{2ncomp} below for the meaning of (ii). In the
above statement, the main implication is $(i) \Rightarrow (iii)$.
The starting point of its proof is the Arazy-Friedman work
\cite{AF2, AF3} giving a list of all $1$-complemented subspaces of
$S^p(\H,\K)$. In Section 2, we give some background on this
classification and some preliminary results, as well as a brief
account on the matricial structure of Schatten spaces and
completely bounded maps on their subspaces. The strategy to prove
$(ii) \Rightarrow (iii)$ consists in taking any $1$-complemented
$X\subset S^p(\H,\K)$ from the Arazy-Friedman list, to exhibit a
canonical contractive projection onto $X$, and to determine
whether that projection is completely contractive (or
$[2]$-$1$-contractive). This is mostly achieved in Sections 3-5.
Theorem \ref{1Main} is eventually proved in Section 6.

Let $n\geq 1$ be an integer, let $\C_{2n}$ be the Clifford algebra
generated by a collection $(\omega_1,\ldots,\omega_{2n})$ of
Fermions, and let $F_n\subset \C_{2n}$ be the linear span of
$\{1,\omega_1,\ldots,\omega_{2n},\omega_1\cdots\omega_{2n}\}$.
Then let $\tau\colon F_n\to F_n$ be the linear mapping such that
$\tau(\omega_1\cdots\omega_{2n})=-\omega_1\cdots\omega_{2n}$ and
$\tau$ is the identity on the linear span of
$\{1,\omega_1,\ldots,\omega_{2n}\}$. The space $F_n$ is a Cartan
factor of type 4 and $\tau$ is a triple isomorphism. This
`transpose map' plays a key role in the study of $1$-complemented
subspaces of $S^p(\H,\K)$ (see Section 5). In Section 7, we
investigate further properties of $\tau$ in the framework of
operator space theory. First we show that $\cbnorm{\tau}=(n+1)/n$.
Then let $F_n^p\subset L^p(\C_{2n})$ be the space $F_n$ regarded
an a subspace of the noncommutative $L^p$-space associated to
$\C_{2n}$. We determine when $\tau\colon F_n^p\to F_n^p$ is
completely contractive (it depends on $n$ and $p$), and we give
applications and complements.

We refer the reader to \cite{NR1,NR2,NRR} for some work on
contractive and completely contractive projections on some Cartan
factors, which is somehow related to the present paper. We also
mention the Ng-Ozawa paper \cite{NO} for a description of the
completely $1$-complemented subspaces of noncommutative
$L^1$-spaces.

\medskip
\section{Background on complete boundedness and $1$-complemented subspaces}
We start with some preliminary facts concerning completely bounded
maps on Schatten spaces and their subspaces. Let $1\leq p<\infty$,
let $\H,\H',\K,\K'$ be Hilbert spaces and consider subspaces
$X\subset S^p(\H,\K)$ and $Y\subset S^p(\H',\K')$. For any index
set $I$, we let
$$
S^p_I\mathop{\otimes}\limits^p X\,\subset\,
S^p(\ell^2_I(\H),\ell^2_I(\K))
$$
denote the completion of $S^p_I\otimes X$ induced by the embedding
of $S^p_I\otimes S^p(\H,\K)$ into the space
$S^p(\ell^2_I(\H),\ell^2_I(\K))$.

Note that for any integer $n\geq 1$, $S^p_n\otimes X$ coincides
with the space of all $n\times n$ matrices with entries in $X$.
Let $u\colon X\to Y$ be a bounded linear map. We set
\begin{equation}\label{2n}
\nnorm{u}\,=\,\bignorm{I_{S^p_n}\otimes u\colon S^p_n
\mathop{\otimes}\limits^p X\longrightarrow S^p_n
\mathop{\otimes}\limits^p Y}
\end{equation}
for any $n\geq 1$, and we say that $u$ is $[n]$-contractive if
$\nnorm{u}\leq 1$. This is equivalent to
\begin{equation}\label{2nbis}
\bignorm{[u(x_{ij})]}_{S^p(\ell^2_n(\footnotesize{\H'}),
\ell^2_n(\footnotesize{\K'}))}\leq
\bignorm{[x_{ij}]}_{S^p(\ell^2_n(\footnotesize{\H}),
\ell^2_n(\footnotesize{\K}))}, \qquad x_{ij}\in X,\ 1\leq i,j\leq
n.
\end{equation}
Next  we set
\begin{equation}\label{2cb}
\cbnorm{u}=\sup_{n\geq 1}\nnorm{u}.
\end{equation}
By definition, $u$ is completely bounded if $\cbnorm{u}<\,\infty$,
and we say that $u$ is a complete contraction (or is completely
contractive) if $\cbnorm{u}\leq 1$. Also we say that $u$ is a
complete isometry if $I_{S^p_n}\otimes u$ is an isometry for any
$n\geq 1$.

The above definitions come from Pisier's fundamental work
\cite{P1} and we wish to point out that they are consistent with
the usual terminology of operator space theory. Indeed, assume
that Schatten spaces are equipped with their `natural' operator
space structure introduced in Pisier's memoir. Then equip any
subspace of a Schatten space with the inherited structure. With
these conventions it is easy to check that the spaces $S^p_n
\mathop{\otimes}\limits^p X$ and $S^p_n \mathop{\otimes}\limits^p
Y$ coincide with the operator space valued Schatten spaces
$S^p_n[X]$ and $S^p_n[Y]$ from \cite[Chapter 1]{P1}. Hence it
follows from \cite[Lem. 1.7]{P1} that the definitions of $\nnorm{\
}$ and $\cbnorm{\ }$ given by (\ref{2n}) and (\ref{2cb}) coincide
with the ones obtained by regarding $X\subset S^p(\H,\K)$ and
$Y\subset S^p(\H',\K')$ as operator spaces. We shall not use much
of operator space theory and we refer the interested reader to
\cite{P2}, \cite{ER} or \cite{Pa} for basic definitions and
background.

\begin{definition}\label{2ncomp} Let $n\geq 1$ be an integer.
We say that $X\subset S^p(\H,\K)$ is $[n]$-$1$-complemented if $X$
is the range of an $[n]$-contractive projection $P\colon
S^p(\H,\K)\to S^p(\H,\K)$. Then we say that $X$ is completely
$1$-complemented if it is the range of a completely contractive
projection $P\colon S^p(\H,\K)\to S^p(\H,\K)$.
\end{definition}

Note that $S^2(\H,\K)$ is `homogeneous', that is, any bounded
linear map $u\colon S^2(\H,\K)\to S^2(\H,\K)$ is automatically
completely bounded, with $\cbnorm{u}=\norm{u}$ (see \cite[Chap.
7]{P2}). Consequently, any $X\subset S^2(\H,\K)$ is completely
$1$-complemented. Thus we will only focus on $1\leq
p\not=2<\infty$ in the sequel.

\bigskip
We say that $X\subset S^p(\H,\K)$ and $Y\subset S^p(\H',\K')$ are
equivalent, and we write
$$
X\sim Y,
$$
if there exist two partial isometries $U\colon \H'\to \H$ and
$V\colon \K'\to \K$ such that
\begin{equation}\label{2Equiv1}
X=VYU^*\qquad\hbox{and}\qquad Y=V^*XU.
\end{equation}
Note that if $X=VYU^*$, then $Y=V^*XU$ if and only if
$y=V^*VyU^{*}U$ for any $y\in Y$, if and only if the mapping
$y\mapsto VyU^*$ is one-to-one on $Y$.

\begin{lemma}\label{2Mult}
Let $\H,\H',\K,\K'$ be Hilbert spaces, and let $W_1\colon \H\to
\H'$ and $W_2\colon \K\to \K'$ be two contractions. Then the
linear mapping $S^p(\H',\K')\to S^p(\H,\K)$ taking any $z\in
S^p(\H',\K')$ to $W_2^*zW_1$ is a complete contraction.
\end{lemma}

\begin{proof} This is clear using (\ref{2nbis}).
\end{proof}

\begin{lemma}\label{2Equiv2}
Assume that $X\subset S^p(\H,\K)$ and $Y\subset S^p(\H',\K')$ are
equivalent. Then $X$ and $Y$ are completely isometric and for any
$n\geq 1$, $X$ is $[n]$-$1$-complemented in $S^p(\H,\K)$ if and
only if $Y$ is $[n]$-$1$-complemented in $S^p(\H',\K')$. Also, $X$
is completely $1$-complemented in $S^p(\H,\K)$ if and only if $Y$
is completely $1$-complemented in $S^p(\H',\K')$.
\end{lemma}

\begin{proof} Lemma \ref{2Mult} ensures that $y\mapsto
VyU^*$ is a complete isometry from $Y$ onto $X$. Now suppose that
$P\colon S^p(\H,\K)\to S^p(\H,\K)$ is a contractive projection
whose range is equal to $X$, and that $X$ and $Y$ satisfy
(\ref{2Equiv1}). Then the mapping $Q\colon S^p(\H',\K')\to
S^p(\H',\K')$ defined by
$$
Q(z)= V^*P(VzU^*)U,\qquad z\in S^p(\H',\K'),
$$
is a contractive projection whose range is equal to $Y$. Moreover
it follows from Lemma \ref{2Mult} that $\nnorm{Q}\leq\nnorm{P}$
for any integer $n\geq 1$. This implies the second part of the
statement.
\end{proof}

\begin{remark}\label{2Equiv3} Although it is not appearent in the
notation, the property $X\sim Y$ depends on the embeddings
$X\subset S^p(\H,\K)$ and $Y\subset S^p(\H',\K')$, and not only on
the operator space structures of $X$ and $Y$. Namely, $X$ and $Y$
may be completely isometric without being equivalent. This
subtlety should not lead to any confusion, since the embeddings
considered for various spaces studied below will be clear from the
context. Note also that if we have Hilbert spaces $H\subset \H$
and $K\subset \K$, then $S^p(H,K)$ regarded as a subspace of
$S^p(\H,\K)$ is equivalent to $S^p(H,K)$ regarded as a subspace of
itself.
\end{remark}

\bigskip
In the second part of this section, we review the classification
of $1$-complemented subspaces of $S^p(\H,\K)$ obtained by
Arazy-Friedman \cite{AF2,AF3}. We fix some $1\leq p\not=2<\infty$
throughout.

Let $X_1,X_2\subset S^p(\H,\K)$ be two subspaces. We say that
$X_1$ and $X_2$ are orthogonal if
$$
x_1^*x_2=0\qquad\hbox{and}\qquad  x_1x_2^*=0,\qquad x_1\in X_1,\
x_2\in X_2.
$$
As observed in \cite[p. 18]{AF3}, this is equivalent to the
identity
\begin{equation}\label{2Indec}
\norm{x_1 +x_2}^p=\norm{x_1}^p + \norm{x_2}^p,\qquad x_1\in X_1,\
x_2\in X_2.
\end{equation}
Also it is easy to check (left to the reader) that this is
equivalent to the existence of orthogonal decompositions
$\H=H_1\mathop{\oplus}\limits^2 H_2$ and
$\K=K_1\mathop{\oplus}\limits^2 K_2$ such that $X_i\subset
S^p(H_i,K_i)$ for $i=1,2$. Consequently, if $(X_\alpha)_\alpha$ is
a family of pairwise orthogonal subspaces of $S^p(\H,\K)$, the
closed subspace of $S^p(\H,\K)$ generated by the $X_\alpha$'s is
equal to their $p$-direct sum $\mathop{\oplus}\limits^p_\alpha
X_\alpha$. Furthermore we have
\begin{equation}\label{2Dec}
S^p_n\mathop{\otimes}\limits^p\bigl(
\mathop{\oplus}\limits^p_\alpha X_\alpha\bigr)\,=\,
\mathop{\oplus}\limits^p_\alpha\bigl(
S^p_n\mathop{\otimes}\limits^p X_\alpha\bigr)
\end{equation}
for any $n\geq 1$.

We say that $X\subset S^p(\H,\K)$ is indecomposable if it cannot
be written as the direct sum of two non trivial orthogonal
subspaces. According to \cite[Prop. 2.2]{AF3}, any subspace $X$ of
$S^p(\H,\K)$ is equal to a direct sum
$\mathop{\oplus}\limits^p_\alpha X_\alpha$ of pairwise orthogonal
indecomposable subspaces. For that reason we will concentrate on
indecomposable subspaces in the rest of this section and in the
next three sections. We note that if $X$ and $Y$ are two subspaces
of some $S^p$-spaces, and if $X$ and $Y$ are isometric, then $X$
is indecomposable if and only if $Y$ is indecomposable. Indeed,
this follows from (\ref{2Indec}).

\bigskip
For any two index sets $I$ and $J$, we regard elements of
$S^p_{I,J}$ as scalar matrices $[t_{ij}]_{i\in I,j\in J}$ in the
usual way. Then we let $\sigma\colon S^p_{I,J}\to S^p_{J,I}$ be
the transpose map, defined by
$$
\sigma\bigl([t_{ij}]\bigr)\,=\, [t_{ji}].
$$
This is an isometry. In the case $J=I$, we let
$$
\S^p_I=\{w\in S^p_I\,:\, \sigma(w)=w\}\qquad\hbox{and}\qquad
\A^p_I=\{w\in S^p_I\,:\, \sigma(w)=-w\}
$$
be the spaces of symmetric and anti-symmetric matrices,
respectively.

It is clear that $\S^p_I$ and $\A^p_I$ are $1$-complemented
subspaces of $S^p_I$. Indeed,
\begin{equation}\label{2sym}
P_s=\frac{1}{2}(Id +\sigma)\qquad\hbox{and}\qquad
P_a=\frac{1}{2}(Id -\sigma)
\end{equation}
are contractive projections whose range are equal to $\S^p_I$ and
$\A^p_I$ respectively. Likewise, for any operator $a\in S^p(H)$ in
some $S^p$-space, the two spaces $\S^p_I\otimes a$ and
$\A^p_I\otimes a$ are $1$-complemented subspaces of
$S^p(\ell^2_I(H))$.

\begin{definition}\label{2Sym}
We say that $X\subset S^p(\H,\K)$ is a space of symmetric matrices
(resp. of anti-symmetric matrices) if it is equivalent to a space
of the form $\S^p_I\otimes a$ (resp. $\A^p_I\otimes a$), where $I$
is an index set and $a\in S^p(H)$.
\end{definition}

Let $a_1\in S^p(H_1)$ and $a_2\in S^p(H_2)$ be two operators, and
consider the spaces
$$
Y_1 = S^p_{I,J}\otimes a_1\subset S^p(\ell^2_J(H_1),\ell^2_I(H_1))
\qquad\hbox{and}\qquad Y_2 =  S^p_{J,I}\otimes a_2 \subset
S^p(\ell^2_I(H_2),\ell^2_J(H_2)),
$$
as well as
\begin{equation}\label{2Rectform}
Y=\{(w\otimes a_1, \sigma(w)\otimes a_2)\, :w\in S^p_{I,J}
\}\subset Y_1 \mathop{\oplus}\limits^p Y_2.
\end{equation}
In this definition we assume that $(a_1,a_2)\not=(0,0)$, excluding
the trivial case $Y=\{0\}$. However $a_1$ or $a_2$ can be equal to
$0$.

The space $Y$ is $1$-complemented in $Y_1 \mathop{\oplus}\limits^p
Y_2$. To check this fact, and also for further purposes, it is
convenient to use matrix notation. In the sequel, for any $z_1\in
S^p_{I,J},\ z_2\in S^p_{J,I}$ we identify $(z_1\otimes a_1,
z_2\otimes a_2)\in Y_1 \mathop{\oplus}\limits^p Y_2$ with the
$2\times 2$ diagonal matrix
$$
\left[\begin{array}{cc} z_1 & 0\\ 0 & z_2\end{array}\right].
$$
We may assume that $\norm{a_1}_p^p +\norm{a_2}_p^p=1$, and we let
$t=\norm{a_1}_p^p$. Then $\norm{z_1\otimes
a_1}=\norm{z_1}\norm{a_1}= t^{\frac{1}{p}}\norm{z_1}$, whereas
$\norm{z_2\otimes a_2}= (1-t)^{\frac{1}{p}}\norm{z_2}$. Hence the
norm on $Y_1 \mathop{\oplus}\limits^p Y_2$ in the above
identification is given by
$$
\left\Vert\left[\begin{array}{cc} z_1 & 0\\ 0 &
z_2\end{array}\right]\right\Vert =\bigl(t\norm{z_1}^p +
(1-t)\norm{z_2}^p\bigr)^{\frac{1}{p}}.
$$
Furthermore
$$
Y\,=\,\biggl\{\left[\begin{array}{cc} w & 0\\ 0 &
\sigma(w)\end{array}\right]\, :\, w\in S^p_{I,J}\biggr\}.
$$
Let $P\colon Y_1 \mathop{\oplus}\limits^p Y_2\to Y_1
\mathop{\oplus}\limits^p Y_2$ be the linear mapping defined by
\begin{equation}\label{2Proj}
P\left(\left[\begin{array}{cc} z_1 & 0\\ 0 &
z_2\end{array}\right]\right)\,=\, \left[\begin{array}{cc} t z_1 +
(1-t)\sigma^{-1}(z_2) & 0\\ 0 & t \sigma(z_1) +
(1-t)z_2\end{array}\right].
\end{equation}
It is plain that $P$ is a projection onto $Y$. Moreover by
convexity we have
\begin{align*}
\norm{tz_1 + (1-t)\sigma^{-1}(z_2)}^p & \leq\bigl(t \norm{z_1} +
(1-t)\norm{\sigma^{-1} (z_2)}\bigr)^p\\ & \leq t \norm{z_1}^p +
(1-t)\norm{\sigma^{-1}(z_2)}^p\leq
\left\Vert\left[\begin{array}{cc} z_1 & 0\\ 0 &
z_2\end{array}\right]\right\Vert^p
\end{align*}
for any $z_1\in S^p_{I,J},\, z_2\in S^p_{J,I}$. Likewise,
$$
\norm{t \sigma(z_1) + (1-t) z_2}\leq
\left\Vert\left[\begin{array}{cc} z_1 & 0\\ 0 &
z_2\end{array}\right]\right\Vert,
$$
which shows that the projection $P$ is contractive. This implies
that $Z$ is $1$-complemented in the $p$-direct sum of
$S^p(\ell^2_J(H_1),\ell^2_I(H_1))$ and
$S^p(\ell^2_I(H_2),\ell^2_I(H_2))$, and hence in the $S^p$-space
$$
S^p\bigl( \ell^2_J(H_1)\mathop{\oplus}\limits^2 \ell^2_I(H_2),
\ell^2_I(H_1) \mathop{\oplus}\limits^2 \ell^2_I(H_2)\bigr).
$$

\begin{definition}\label{2Rect}
We say that $X\subset S^p(\H,\K)$ is a space of rectangular
matrices if it is equivalent to a space $Y$ of the form
(\ref{2Rectform}).
\end{definition}

We now turn to the construction of operator spaces acting on
anti-symmetric Fock spaces. We refer the reader to \cite{BR,PR}
for general information on these spaces. Let $n\geq 1$ be an
integer. For any $k=0,\ldots, n$, we let $\Lambda_{n,k}$ denote
the $k$-fold anti-symmetric tensor product of the Hilbert space
$\ell^2_n$, equipped with the canonical inner product given by
$$
\bigl\langle \xi_1\wedge\cdots \wedge \xi_k\, ,\,
\xi'_1\wedge\cdots\wedge \xi'_k\bigr \rangle\, =\, {\rm
det}\bigl[\langle \xi_i, \xi'_j\rangle\bigr],\qquad \xi_i,\
\xi'_j\in \ell^2_n.
$$
By convention, $\Lambda_{n,0}=\Cdb$. We let $\Omega$ be a
particular unit element of $\Lambda_{n,0}$, which is called the
vacuum vector. Then the anti-symmetric Fock space over $\ell^2_n$
is the Hilbertian direct sum
$$
\Lambda_n\, =\, {\mathop{\oplus}\limits^2_{0\leq k\leq
n}}\Lambda_{n,k}.
$$
Throughout we let $(e_1,\ldots,e_n)$ denote the canonical basis of
$\ell^2_n$ and we let $\P_n$ be the set of all subsets of
$\{1,\ldots,n\}$. Let $A\in\P_n$ with cardinal $\vert A\vert =k$,
and let $1\leq j_1<\cdots<j_k\leq n$ be the increasing enumeration
of the elements of $A$. Then we set
\begin{equation}\label{2EA}
e_A=e_{j_1}\wedge\cdots\wedge e_{j_k}.
\end{equation}
By convention, $e_{\emptyset}=\Omega$. Clearly the system $\{e_A\,
:\, \vert A\vert =k\}$ is an orthonormal basis of $\Lambda_{n,k}$.
We will call it `canonical' in the sequel. Note that ${\rm
dim}(\Lambda_{n,k})=\binom{n}{k}$ and that ${\rm
dim}(\Lambda_{n})=2^n$.

For any $1\leq j\leq n$, we let
$$
c_{n,j}\colon\Lambda_n\longrightarrow \Lambda_n
$$
be the so-called creation operator defined by letting
$c_{n,j}(\Omega)=e_j$, and
$$
c_{n,j}(\xi_1\wedge\cdots \wedge \xi_k) =
e_j\wedge\xi_1\wedge\cdots \wedge \xi_k,\qquad
\xi_1,\ldots,\xi_k\in\ell^2_n.
$$
Next we denote by $P_n\colon \Lambda_n\to \Lambda_n$ the
orthogonal projection onto the space
$$
\Lambda_{n}^{\rm even}\,=\,
{\mathop{\oplus}\limits_{\substack{0\leq k\leq n\\ k\
\hbox{\footnotesize even}}}}\Lambda_{n,k}
$$
generated by tensor products of even rank. Following \cite[p.
24]{AF2}, we let
$$
x_{n,j}=c_{n,j}P_n\qquad\hbox{and}\qquad \widetilde{x}_{n,j} =
c_{n,j}^* P_n
$$
be the restrictions of $c_{n,j}$ and $c_{n,j}^*$ to
$\Lambda_{n}^{\rm even}$ for any $1\leq j\leq n$. Then we set
$$
AH_n\,=\,{\rm Span}\{x_{n,j},\widetilde{x}_{n,j}\, :\,
j=1,\ldots,n\}.
$$
This is a $2n$-dimensional operator space. Next we let
$$
BH_n=\{x^*\, :\, x\in AH_n\}\,=\,{\rm Span}\{x_{n,j}^*,
\widetilde{x}_{n,j}^*\, :\, j=1,\ldots,n\}
$$
be the adjoint space of $AH_n$. Note that $c_{n,j}=P_n c_{n,j} +
c_{n,j}P_n$. Consequently,
$$
\widetilde{x}_{n,j}^* = c_{n,j}(Id-P_n)\qquad\hbox{and}\qquad
x_{n,j}^* = c_{n,j}^{*}(Id-P_n)
$$
are the restrictions of $c_{n,j}$ and $c_{n,j}^*$ to the space
$$
\Lambda_n^{\rm odd}\,=\,\Lambda_n\ominus\Lambda_n^{\rm
even}\,=\,{\mathop{\oplus}\limits_{\substack{0\leq k\leq n\\ k\
\hbox{\footnotesize odd}}}}\Lambda_{n,k}
$$
generated by the tensor products of odd rank.

In the sequel we let $AH_n^p$ and $BH_n^p$ denote the spaces
$AH_n$ and $BH_n$ respectively, regarded as subspaces of
$S^p(\Lambda_n)$.

Let $\kappa\colon AH_n\to BH_n$ be the exchange map defined by
letting
\begin{equation}\label{2Kappa}
\kappa(x_{n,j})={\widetilde{x}_{n,j}}^*\qquad\hbox{and}\qquad
\kappa(\widetilde{x}_{n,j})= x_{n,j}^*,\qquad j=1,\ldots, n.
\end{equation}
It follows from the calculations in \cite[Chap. 2]{AF2} that
$\kappa$ is an isometry from $AH^p_n$ onto $BH^p_n$. (An explicit
proof of this fact will be given in Section 5, see Remark
\ref{5Final}). For any operators $a_1\in S^p(H_1)$ and $a_2\in
S^p(H_2)$, with $(a_1,a_2)\not= (0,0)$, we will consider
\begin{equation}\label{2Spinform}
Z=\bigl\{(x\otimes a_1, \kappa(x)\otimes a_2)\, :\, x\in
AH_n^p\bigr\}\subset S^p(\Lambda_n \mathop{\otimes}\limits^2 H_1)
\mathop{\oplus}\limits^p S^p(\Lambda_n \mathop{\otimes}\limits^2
H_2).
\end{equation}
According to \cite[Prop. 2.9]{AF2}, this space is $1$-complemented
in the $p$-direct sum of $S^p(\Lambda_n \mathop{\otimes}\limits^2
H_1)$ and $S^p(\Lambda_n \mathop{\otimes}\limits^2 H_2)$.

Now following \cite[p. 33]{AF2} we consider the
$(2n-1)$-dimensional operator space
$$
DAH_n\,=\,{\rm Span}\{x_{n,n} + \widetilde{x}_{n,n};\
x_{n,j},\widetilde{x}_{n,j}\, :\, j=1,\ldots,n-1\},
$$
and we let $DAH_n^p$ be that space regarded as a subspace of
$S^p(\Lambda_n)$. Then for any $a\in S^p(H)$, the space
$DAH_n^p\otimes a$ is $1$-complemented in $S^p(\Lambda_n
\mathop{\otimes}\limits^2 H)$,  by \cite[Prop. 2.13]{AF2}.

Simple proofs of the above mentioned $1$-complementation results
will be given later on in Section 5, see Remark \ref{5Final}.

\begin{definition}\label{2Spin} Let
$X\subset S^p(\H,\K)$ be a finite dimensional space of dimension
$N\geq 1$. If $N=2n-1$ is odd, we say that $X$ is a spinorial
space if it is equivalent to a space of the form $DAH_n^p\otimes
a$, for some $a\in S^p(H)$. If $N=2n$ is even, we say that $X$ is
a spinorial space if it is equivalent to a space of the form
(\ref{2Spinform}).
\end{definition}

See the end of this section for more on this terminology.

\bigskip
We shall now define a class of finite dimensional Hilbert spaces
which are $1$-complemented subspaces of $S^p$. Let $1\leq k\leq
n$. It is clear that the creation operators $c_{n,j}$ map
$\Lambda_{n,k-1}$ into $\Lambda_{n,k}$. For any $j=1,\ldots, n$,
we let
$$
c_{n,j,k}\colon \Lambda_{n,k-1}\longrightarrow\Lambda_{n,k}
$$
be the restriction of $c_{n,j}$ to $\Lambda_{n,k-1}$. A quick
examination of the definition of the $c_{n,j}$'s shows that the
matrix of $c_{n,j,k}$ in the canonical bases of $\Lambda_{n,k-1}$
and $\Lambda_{n,k}$ has its entries in $\{-1,0,1\}$, with at most
one non zero element on each row and on each column. Moreover the
$\pm 1$ entries appear exactly $\binom{n-1}{k-1}$ times. Hence
$\norm{c_{n,j,k}}^p_p= \binom{n-1}{k-1}$ for any $j$. We let
$$
H_{n,k}\,=\,{\rm Span}\{c_{n,j,k}\, : \, j=1,\ldots,n\},
$$
and we let $H_{n,k}^p$ be that space regarded as a subspace of
$S^p(\Lambda_n)$. By \cite[Chap. 7]{AF3}, $H_{n,k}^p$ is an
$n$-dimensional Hilbert space. More precisely, the linear mapping
\begin{equation}\label{2Varphi}
\varphi_k\colon\ell^{2}_n\longrightarrow H_{n,k}^p,\qquad
\varphi_k(e_j)= c_{n,j,k},\qquad 1\leq j\leq n,
\end{equation}
is a multiple of an isometry, i.e.
$\bigl[\binom{n-1}{k-1}\bigr]^{-\frac{1}{p}} \varphi_k$ is an
isometry. We refer the reader to \cite{NR1, NR2, NRR} for more on
these Hilbert spaces and their operator space properties.

For any operators $a_1\in S^p(H_1),\ldots, a_n\in S^p(H_n)$, with
$(a_1,\ldots,a_n)\not=(0,\ldots,0)$, we will consider the space
\begin{equation}\label{2Hilbertform}
E\,=\, \bigl\{(\varphi_1(s)\otimes a_1,\ldots,\varphi_n(s)\otimes
a_n)\, : s\in\ell^{2}_n\bigr\}\subset
\mathop{\oplus}\limits^{p}_{1\leq k\leq n}
S^p(\Lambda_n\mathop{\otimes}\limits^2 H_k).
\end{equation}
Clearly $E$ is a Hilbert space. Indeed if we assume (after
normalisation) that
$$
\sum_k\tbinom{n-1}{k-1}\norm{a_k}^p_p =1,
$$
then the linear mapping $\ell^2_n\to E$ taking any $s\in\ell^2_n$
to $(\varphi_1(s)\otimes a_1,\ldots,\varphi_n(s)\otimes a_n)$ is
an isometry. According to \cite[Prop. 2.5]{AF2}, the space $E$ is
$1$-complemented.

\begin{theorem}\label{2AF1}\,{\rm (Arazy-Friedman)}
Let $\H,\K$ be Hilbert spaces, and let $X\subset S^p(\H,\K)$ be an
indecomposable subspace, with $1\leq p\not= 2<\infty$. The
following are equivalent.
\begin{itemize}
\item [(i)] $X$ is $1$-complemented in $S^p(\H,\K)$. \item [(ii)]
$X$ is either a space of symmetric matrices, or a space of
anti-symmetric matrices (in the sense of Definition \ref{2Sym}),
or a space of rectangular matrices (in the sense of Definition
\ref{2Rect}), or a spinorial space (in the sense of Definition
\ref{2Spin}) of dimension $\geq 5$, or a finite dimensional
Hilbertian space equivalent to a space of the form
(\ref{2Hilbertform}).
\end{itemize}
\end{theorem}

By Lemma \ref{2Equiv2} and the results we have recorded along this
section, all the spaces in the list (ii) are $1$-complemented. The
hard implication `$(i)\Rightarrow (ii)$' is proved in \cite[Chap.
7]{AF3} in the case $p>1$ and in \cite[Chap. 5]{AF2} in the case
$p=1$.

\bigskip
After reducing to the case of indecomposable spaces, the proof of
Theorem \ref{1Main} will mainly consist in showing that the spaces
in the list (ii) above are not completely $1$-complemented, except
the ones which are equivalent to some $S^p_{I,J}\otimes a$. This
will be achieved in the next three sections.

It should be noticed that the classes of $1$-complemented
subspaces considered above do not exclude each other. For
instance, the Hilbert space $S^p_{n,1}$ is equivalent to
$H_{n,1}^p$, whereas $S^p_{1,n}$ is equivalent to $H_{n,n}^p$. On
the other hand, $AH_1^p=\ell^p_2$ and it follows from \cite[Chap.
2]{AF2} that $AH_2^p$ is equivalent to $S_2^p$, $AH_3^p$ is
equivalent to $\A_4^p$, $DAH_2^p$ is equivalent to $\S_2^p$ and
$\A_3^p$ is equivalent to $H_{3,2}^p$.

\begin{remark}\label{2Uniqueness}
Suppose that $p>1$. If $X\subset S^p(\H,\K)$ is $1$-complemented,
then the  contractive projection $P\colon S^p(\H,\K)\to
S^p(\H,\K)$ whose range is equal to $X$ is unique (see \cite[Prop.
1.2]{AF3}).

This uniqueness property is false in the case $p=1$ (see e.g.
\cite[p. 36]{AF2}). However a similar result holds true, as
follows. Let $X\subset S^1(\H,\K)$ be a subspace and let $r\in
B(\H)$ and $\ell\in B(\K)$ be the smallest orthogonal projections
such that $\ell xr=x$ for any $x\in X$.  Let $H\subset\H$ and
$K\subset\K$ be the ranges of $r$ and  $\ell$, respectively. Thus
$X\subset S^1(H,K)$, and $H,K$ are the smallest subspaces of
$\H,\K$ with that property. We say that $X$ is nondegenerate if
$H=\H$ and $K=\K$. It is proved in \cite[Th. 2.15]{AF2} that if
$X$ is $1$-complemented and nondegenerate, then the  contractive
projection $P$ on $S^1(\H,\K)$ with range  equal to $X$ is unique.

Note that $X$ regarded as a subspace of $S^1(\H,\K)$ is equivalent
to $X$ regarded as a subspace of $S^1(H,K)$. Thus if we wish to
determine whether $X$ is $[n]$-$1$-complemented (for some $n\geq
1$), there is no loss of generality in assuming that $X$ is
nondegenerate.
\end{remark}

\bigskip
We end this section with some terminology and notions which play a
central role in the work of Arazy-Friedman \cite{AF2, AF3}, and
some basic facts.

Let $\H,\K$ be Hilbert spaces and let $X\subset B(\H,\K)$ be a
closed subspace. By definition $X$ is a $JC^*$-triple if $xx^*x$
belongs to $X$ for any $x\in X$. Next a linear map $u\colon X\to
Y$ between two $JC^*$-triples $X$ and $Y$ is called a triple
homomorphism if $u(xx^*x)=u(x)u(x)^*u(x)$ for any $x\in X$. If $u$
is one-to-one, we say that $u$ is a triple monomorphism. If
further $u$ is a bijection, then $u^{-1}$ also is a triple
homomorphism and we say that $u$ is a triple isomorphism in this
case. It is well-known that a bijection $u\colon X\to Y$ between
two $JC^*$-triples is an isometry if and only if it is a triple
isomorphism (see \cite{H}). We say that $X$ and $Y$ are triple
equivalent if there is a triple isomorphism from $X$ onto $Y$.

We now turn to Cartan factors of types 1-4. We mainly follow
\cite{H} (see also \cite{NR1}). By definition, a Cartan factor of
type 1 is a $JC^*$-triple which is triple equivalent to some
$B(\H,\K)$, where $\H,\K$ are Hilbert spaces. Next let $\H$ be a
Hilbert space with a distinguished Hilbertian basis, and let
$w\mapsto { }^tw$ denote the associated transpose map on $B(\H)$.
Then the space of anti-symmetric operators
$$
\A(\H)\,=\, \{w\in B(\H)\, :\ { }^tw =-w\}
$$
is a $JC^*$-triple, and we call Cartan factor of type 2 any
$JC^*$-triple which is triple equivalent to some $\A(\H)$.
Likewise, the  space of symmetric operators
$$
\S(\H)\,=\, \{w\in B(\H)\, :\ { }^tw = w\}
$$
is a $JC^*$-triple, and we call Cartan factor of type 3 any
$JC^*$-triple which is triple equivalent to some $\S(\H)$. Lastly,
let $X\subset B(\H)$ be a closed subspace such that $x^*\in X$ for
any $x\in X$ and $x^2$ is a scalar multiple of the identity
operator for any $x\in X$. Then $X$ is a $JC^*$-triple, and we
call Cartan factor of type 4 any $JC^*$-triple which is triple
equivalent to such a space.

Let $n\geq 2$ be an integer. An $n$-tuple $(s_1,\ldots, s_n)$ of
operators in some $B(\H)$ is called a spin system if each $s_j$ is
a selfadjoint unitary and
$$
s_{j}s_{j'} + s_{j'}s_{j}=0,\qquad 1\leq j\not= j'\leq n.
$$
In this case, the $n$-dimensional space
$$
X={\rm Span}\{s_1,\ldots,s_n\}\subset B(\H)
$$
is a Cartan factor of type 4.

Let $w_1,\ldots,w_n$ be the operators on $\Lambda_n$ defined by
$$
\omega_j=c_{n,j}+c_{n,j}^{*}, \qquad j=1,\ldots,n.
$$
These operators are called Fermions and they form a spin system
(see e.g. \cite{BR}). Hence their linear span is an
$n$-dimensional Cartan factor of type 4. It is well-known that all
$n$-dimensional Cartan factors of type 4 are mutually triple
equivalent. Thus the space ${\rm
Span}\{\omega_1,\ldots,\omega_n\}$ is actually a model for such
spaces.

It turns out that the spaces $AH_n$ and $DAH_n$ considered in this
section are Cartan factors of type 4. This will be implicitly
shown along the proof of Theorem \ref{5Spin}. We refer the reader
to \cite{AF3} for details on this, and for a deeper analysis of
the relationship between $1$-complemented subspaces of
$S^p$-spaces and Cartan factors.

In the framework of operator space theory, it is natural to wonder
whether a triple isomorphism between Cartan factors is necessarily
a complete isometry, that is, if the identification of Cartan
factors in the category of $JC^*$-triples coincide with their
identification as operator spaces. This is not always the case.
Indeed, if $\H$ or $\K$ has dimension $\geq 2$, any transposition
map $B(\H,\K)\to B(\K,\H)$ is a triple isomorphism which is not
completely isometric. This question is more delicate for Cartan
factors of types 2-4 and will be discussed further in Section 7
(see in particular Proposition \ref{7Structures} and Remark
\ref{7Identif}).

\medskip
\section{Elementary computations}
In this section we will treat rectangular matrices, symmetric
matrices, and anti-symmetric matrices. We will only need
elementary matrix computations.

\begin{proposition}\label{3Rect}
Let $X\subset S^p(\H,\K)$ be a space of rectangular matrices, and
assume that $1\leq p\not= 2<\infty$.
\begin{itemize} \item [(1)] If $X$ is $[2]$-$1$-complemented in $S^p(\H,\K)$,
then there exist two index sets $I,J$, a Hilbert space $H$ and  an
operator $a\in S^p(H)$ such that $X$ is equivalent to
$S^p_{I,J}\otimes a$. \item [(2)] If $X$ is equivalent to a space
of the form $S^p_{I,J}\otimes a$, then $X$ is completely
$1$-complemented in $S^p(\H,\K)$.
\end{itemize}
\end{proposition}

\begin{proof} Part (2) is obvious by Lemma \ref{2Equiv2}.
To prove (1), it suffices by Lemma \ref{2Equiv2} again to consider
index sets $I, J$ with $(I,J)\not= (1,1)$ and to show that for any
operators $a_1\in S^p(H_1)$, $a_2\in S^p(H_2)$, the space defined
by (\ref{2Rectform}) is $[2]$-$1$-complemented only if $a_1=0$ or
$a_{2}=0$. We will use matrix notation as introduced in Section 2.
We assume that $\norm{a_1}^p_p+\norm{a_2}^p_p=1$ and we let
$t=\norm{a_1}^p_p$.

Let us write $Y_{I,J}$ for the space defined by (\ref{2Rectform}).
For any $I'\subset I$ and $J'\subset J$, we have a natural
inclusion $Y_{I',J'}\subset Y_{I,J}$. Then $Y_{I',J'}$ is clearly
completely $1$-complemented in $Y_{I,J}$. Hence without loss of
generality we can assume that $I=2$ and $J=1$.

Assume that $Y=Y_{2,1}$ is $[2]$-$1$-complemented. Then by Remark
\ref{2Uniqueness}, the projection $P$ on $Y_1
\mathop{\oplus}\limits^p Y_2$ defined by (\ref{2Proj}) is
$[2]$-contractive. According to(\ref{2Dec}), we have an isometric
identification
$$
S^p_2\mathop{\otimes}\limits^p(Y_1\mathop{\oplus}\limits^p Y_2)
\,=\,\Bigl[\bigl(S^p_2\mathop{\otimes}\limits^p
S^p_{2,1}\bigr)\otimes a_1\Bigr] \mathop{\oplus}\limits^p
\Bigl[\bigl(S^p_2\mathop{\otimes}\limits^p S^p_{1,2}\bigr)\otimes
a_2\Bigr],
$$
and for any $z_1\in S^p_2\mathop{\otimes}\limits^p S^p_{2,1}
\simeq S^p_{4,2}$ and $z_2\in S^p_2\mathop{\otimes}\limits^p
S^p_{1,2} \simeq S^p_{2,4}$, we have
\begin{equation}\label{3Proj}
\bigl(I_{S^p_2}\otimes P\bigr)\left(\left[\begin{array}{cc} z_1 &
0\\ 0 & z_2\end{array}\right]\right)\,=\, \left[\begin{array}{cc}
t z_1 + (1-t)[I_{S^p_2}\otimes\sigma^{-1}](z_2) & 0\\ 0 & t
[I_{S^p_2}\otimes\sigma](z_1) + (1-t)z_2\end{array}\right].
\end{equation}

Assume first that $2<p<\infty$. For any positive angle $\theta>0$,
consider
$$
z_1(\theta)=
\left[\begin{array}{cc} (\cos(\theta))^{\frac{2}{p}} & 0 \\ 0 & 0\\
0 &  (\sin(\theta))^{\frac{2}{p}}\\ 0 & 0
\end{array}\right]\qquad\hbox{and}\qquad
z_2(\theta)=
\left[\begin{array}{cccc} (\cos(\theta))^{\frac{2}{p}} & 0 & 0 & 0\\
0 & 0 &  (\sin(\theta))^{\frac{2}{p}}& 0
\end{array}\right].
$$
Then we have
$$
\norm{z_1(\theta)}_p^p =\norm{z_2(\theta)}^p_p =
\cos(\theta)^{2}+\sin(\theta)^{2}=1,
$$
hence
\begin{equation}\label{3NormOne}
\left\Vert\left[\begin{array}{cc} z_1(\theta) & 0\\ 0 &
z_2(\theta)\end{array}\right]\right\Vert_{S^p_2\mathop{\otimes}\limits^p
(Y_1\mathop{\oplus}\limits^p Y_2)}\, =1.
\end{equation}
Applying the transpose map $\sigma\colon S^p_{2,1}\to S^p_{1,2}$
and its inverse, we find that
$$
[I_{S^p_2}\otimes\sigma](z_1(\theta))= \left[\begin{array}{cccc}
(\cos(\theta))^{\frac{2}{p}} & 0  & 0 &
(\sin(\theta))^{\frac{2}{p}}\\ 0 & 0 & 0 & 0
\end{array}\right]
$$
and that
$$
[I_{S^p_2}\otimes\sigma^{-1}](z_2(\theta))=
\left[\begin{array}{cc} (\cos(\theta))^{\frac{2}{p}} & 0 \\ 0 & 0\\
0 & 0 \\  (\sin(\theta))^{\frac{2}{p}}& 0
\end{array}\right].
$$
Consequently, we have
$$
t z_1(\theta) + (1-t)[I_{S^p_2}\otimes\sigma^{-1}] (z_2(\theta))\,
=\,
\left[\begin{array}{cc} (\cos(\theta))^{\frac{2}{p}} & 0 \\ 0 & 0\\
0 &  t(\sin(\theta))^{\frac{2}{p}}\\
(1-t)(\sin(\theta))^{\frac{2}{p}} & 0
\end{array}\right],
$$
whereas
$$
t [I_{S^p_2}\otimes\sigma](z_1(\theta) ) + (1-t) z_2(\theta) \,
=\, \left[\begin{array}{cccc} (\cos(\theta))^{\frac{2}{p}} & 0 & 0
& t
(\sin(\theta))^{\frac{2}{p}}\\
0 & 0 &  (1-t)(\sin(\theta))^{\frac{2}{p}}& 0
\end{array}\right].
$$
Thus
\begin{align*}
\bignorm{t z_1(\theta) +
(1-t)[I_{S^p_2}\otimes\sigma^{-1}](z_2(\theta))}_p &\,\geq\,
\left\Vert\left[\begin{array}{c} (\cos(\theta))^{\frac{2}{p}}
\\ (1-t)(\sin(\theta))^{\frac{2}{p}}
\end{array}\right]\right\Vert_{S^p_{2,1}}\\ &\,=\,
\bigl((\cos(\theta))^{\frac{4}{p}}
+(1-t)^{2}(\sin(\theta))^{\frac{4}{p}}\bigr)^{\frac{1}{2}}.
\end{align*}
Likewise,
$$
\bignorm{t [I_{S^p_2}\otimes\sigma](z_1(\theta)) +
(1-t)z_2(\theta)}_p\geq \bigl((\cos(\theta))^{\frac{4}{p}}
+t^{2}(\sin(\theta))^{\frac{4}{p}}\bigr)^{\frac{1}{2}}.
$$
Using (\ref{3NormOne}), these estimates imply that
$$
\norm{P}_2^p\geq t \bigl((\cos(\theta))^{\frac{4}{p}}
+(1-t)^{2}(\sin(\theta))^{\frac{4}{p}}\bigr)^{\frac{p}{2}} \, +\,
(1-t)\bigl((\cos(\theta))^{\frac{4}{p}}
+t^{2}(\sin(\theta))^{\frac{4}{p}}\bigr)^{\frac{p}{2}}.
$$
Since $p>2$, we have $\frac{4}{p}<2$. Consequently,
$$
(\cos(\theta))^{\frac{4}{p}} = 1 +
o\bigl(\theta^{\frac{4}{p}}\bigr)\qquad\hbox{and}\qquad
(\sin(\theta))^{\frac{4}{p}} =\theta^{\frac{4}{p}}  +
o\bigl(\theta^{\frac{4}{p}}\bigr)
$$
on a (positive) neighborhood of zero. Hence
\begin{align*}
1=\norm{P}_2^p &\geq t \bigl(1+ (1-t)^{2} \theta^{\frac{4}{p}} +
o\bigl(\theta^{\frac{4}{p}}\bigr)\bigr)^{\frac{p}{2}} \, +\,
(1-t)\bigl(1 + t^{2}\theta^{\frac{4}{p}} +
o\bigl(\theta^{\frac{4}{p}}\bigr)\bigr)^{\frac{p}{2}}\\
&\geq 1 +\frac{p}{2}(t(1-t)^{2}+(1-t)t^2) \theta^{\frac{4}{p}} +
o\bigl(\theta^{\frac{4}{p}}\bigr) \\ & = 1 +\frac{p}{2}t(1-t)
\theta^{\frac{4}{p}} + o\bigl(\theta^{\frac{4}{p}}\bigr).
\end{align*}
This implies that $t(1-t)=0$, that is, $a_1=0$ or $a_{2}=0$.

The case when $1<p<2$ can be treated by duality, or by a direct
similar argument. Indeed, let $p'=(p-1)/p>2$ be the conjugate
number of $p$ and apply (\ref{3Proj}) with
$$
z_1= \left[\begin{array}{cc} (\cos(\theta))^{\frac{2}{p'}} & 0 \\ 0 & 0\\
0 & 0 \\  (\sin(\theta))^{\frac{2}{p'}}& 0
\end{array}\right]\qquad\hbox{and}\qquad z_2 =
\left[\begin{array}{cccc} (\cos(\theta))^{\frac{2}{p'}} & 0  & 0 &
(\sin(\theta))^{\frac{2}{p'}}\\ 0 & 0 & 0 & 0
\end{array}\right].
$$
Then we find that
$$
\left\Vert\left[\begin{array}{cc} z_1 & 0\\ 0 &
z_2\end{array}\right]\right\Vert\, =\,
1+\frac{1}{2}\theta^{\frac{4}{p'}}+
o\bigl(\theta^{\frac{4}{p'}}\bigr)
$$
and that
$$
\left\Vert\bigl(I_{S^p_2}\otimes P\bigr) \left[\begin{array}{cc}
z_1  & 0\\ 0 & z_2 \end{array}\right]\right\Vert\, \geq\, 1
+\frac{1}{p}\,\bigl(t(1-t)^p +
t^p(1-t)\bigr)\theta^{\frac{2p}{p'}} + o\bigl(
\theta^{\frac{2p}{p'}}\bigr).
$$
Since $\frac{2p}{p'}<\frac{4}{p'}$, we deduce that $t=0$ or $t=1$
if $P$ is $[2]$-contractive.

The case $p=1$ has a similar proof, with
$$
z_1= \left[\begin{array}{cc} 1 & 0 \\ 0 & 0\\
0 & 0 \\ \theta & 0
\end{array}\right]\qquad\hbox{and}\qquad z_2 =
\left[\begin{array}{cccc} 1 & 0  & 0 & \theta\\ 0 & 0 & 0 & 0
\end{array}\right].
$$
\end{proof}

\begin{proposition}\label{3Sym} Let $X\subset S^p(\H,\K)$ and assume
that $1\leq p\not= 2<\infty$ and ${\rm dim}(X)>1$. If $X$ is
either a space of symmetric matrices or a space of anti-symmetric
matrices, then $X$ is not $[2]$-$1$-complemented.
\end{proposition}

\begin{proof}
By Lemma \ref{2Equiv2}, it suffices to show that for any $a\in
S^p(H)\setminus\{0\}$, the space $\S_I^p\otimes a$ is not
$[2]$-1-complemented, unless $I=1$, and that the space
$\A_I^p\otimes a$ is not $[2]$-1-complemented, unless $I=1$ or
$2$. Using Remark \ref{2Uniqueness} and an obvious reduction, this
amounts to showing that the contractive projections
$$
P_s\colon S^p_2\longrightarrow S^p_2\qquad\hbox{and}\qquad
P_a\colon S^p_3\longrightarrow S^p_3
$$
given by (\ref{2sym}) are not $[2]$-contractive.

The subspace
$$
\left\{\left[\begin{array}{ccc} 0 & s_{12} & s_{13} \\ s_{21} & 0 & 0\\
s_{31} & 0 & 0
\end{array}\right]\, :\, s_{12}, s_{13}, s_{21},
s_{31} \in\Cdb\right\}\,\subset S^p_3
$$
is completely isometric to $S^p_{2,1}\mathop{\oplus}\limits^p
S^p_{1,2}$, and
$$
P_a\left(\left[\begin{array}{ccc} 0 & s_{12} & s_{13} \\ s_{21} & 0 & 0\\
s_{31} & 0 & 0
\end{array}\right]\right)\, =\,\frac{1}{2}\,
\left[\begin{array}{ccc} 0 & s_{12} - s_{21} & s_{13} -s_{31} \\
s_{21}
-s_{12} & 0 & 0\\
s_{31} -s_{13} & 0 & 0
\end{array}\right].
$$
Hence
$$
\bignorm{I_{S^p_2}\otimes P_a} \geq \bignorm{I_{S^p_2}\otimes P},
$$
where $P$ is the projection defined by (\ref{2Proj}) in the case
when $I=2, J=1$, and $\norm{a_1}=\norm{a_2}$. It follows from the
proof of Proposition \ref{3Rect} that this projection is not
$[2]$-contractive. Thus $P_a$ is not $[2]$-contractive either.

The fact that $P_s$ is not $[2]$-contractive on $S^p_2$ also
follows from the proof of Proposition \ref{3Rect}. We skip the
details.
\end{proof}

\medskip
\section{Finite dimensional Hilbertian subspaces}
In this section we will treat finite dimensional Hilbertian
$1$-complemented subspaces of $S^p$-spaces. We use the notation
introduced in Section 2. Let $n\geq 2$ be an integer. For any
contraction $T\colon\ell^2_n\to \ell^2_n$, we let
$F(T)\colon\Lambda_n\to\Lambda_n$ be the linear mapping defined by
$F(T)\Omega=\Omega$ and
$$
F(T)(\xi_1\wedge\cdots\wedge\xi_k) = T(\xi_1)\wedge\cdots\wedge
T(\xi_k), \qquad \xi_1,\ldots,\xi_k\in\ell^2_n.
$$
It is well-known that $F(T)$ is a contraction (the construction
$T\mapsto F(T)$ is called the second quantization). Note that
$F(T_1T_2)=F(T_1)F(T_2)$ for any two contractions $T_1,T_2$ of
$\ell^{2}_n$. Thus $F(U)$ is a unitary if $U$ is a unitary, and we
have $F(U)^{*}=F(U^*)$ in this case. We will need the following
observation of independent interest.

\begin{lemma}\label{4Quant}
Let $U\colon \ell^2_{n}\to \ell^{2}_n$ be a unitary operator, and
let $\widehat{U}\colon B(\Lambda_n)\to B(\Lambda_n)$ be defined by
$$
\widehat{U}(W)=F(U)WF(U)^{*},\qquad W\in B(\Lambda_n).
$$
Then for any $1\leq k\leq n$, $\widehat{U}(H_{n,k})\subset
H_{n,k}$ and the restriction of $\widehat{U}$ to $H_{n,k}$
coincides with $U$. (More precisely,
$U=\varphi_k^{-1}\widehat{U}\varphi_k$, where
$\varphi_k\colon\ell^{2}_n\to H_{n,k}$ is defined by
(\ref{2Varphi})).
\end{lemma}

\begin{proof}
It is clear that $\widehat{U}$ maps
$B(\Lambda_{n,k-1},\Lambda_{n,k})$ into itself. Assume for
simplicity that $k\geq 2$ (the case $k=1$ is similar). Let $1\leq
j\leq n$, and let $\xi_1,\ldots,\xi_{k-1}\in \ell^2_n$. Then
\begin{align*}
\bigl[\widehat{U}(c_{n,j,k})\bigr](\xi_1\wedge\cdots\wedge
\xi_{k-1})\, & =\, F(U)  c_{n,j,k}
F(U)^{*}(\xi_1\wedge\cdots\wedge \xi_{k-1})\\  & =\, F(U)
c_{n,j,k} \bigl(U^{*}(\xi_1)\wedge\cdots\wedge
U^{*}(\xi_{k-1})\bigr)\\  & =\, F(U) \bigl(e_j\wedge
U^{*}(\xi_1)\wedge\cdots\wedge U^{*}(\xi_{k-1})\bigr)\\  & =\,
U(e_j)\wedge\xi_1\wedge\cdots\wedge \xi_{k-1},
\end{align*}
since $UU^{*}=I_{\ell^{2}_n}$. This yields the result.
\end{proof}

It was proved in \cite[Th. 1]{NR2} that for any $1\leq k\leq n$,
$H_{n,k}\subset B(\Lambda_n)$ is a homogeneous operator space.
Using \cite[Prop. 9.2.1]{P2}, this result readily follows from the
above lemma. The latter implies that $H_{n,k}^p\subset
S^p(\Lambda_n)$ is homogeneous as well.

\begin{proposition}\label{4Hilbert}
Assume that $1\leq p\not=2<\infty$ and let $a_1\in
S^p(H_1),\ldots, a_n\in S^p(H_n)$, with
$(a_1,\ldots,a_n)\not=(0,\ldots,0)$. The $n$-dimensional Hilbert
space $E$ defined by (\ref{2Hilbertform}) is
$[2]$-$1$-complemented if and only if
$$
a_2=\cdots = a_n=0\qquad\hbox{or}\qquad a_1=\cdots = a_{n-1}=0.
$$
\end{proposition}

\begin{proof}
The `if' part is clear. If for example $a_2=\cdots = a_n=0$ then
$$
E=H_{n,1}^p\otimes a_1 \sim S^p_{n,1} \otimes a_1,
$$
hence $E$ is completely complemented.

We shall now prove the `only if' part. We assume that $E$ is
$[2]$-$1$-complemented. We will somehow reduce to the case when
${\rm dim}(E)=2$. Let us apply the second quantization to the
unitary
$$
U=\left[\begin{array}{ccccc} 1 & \ & \ & \ &\  \\ \ & 1 & \ &  (0)
& \ \\ \ & \ & -1 &\ &\ \\ \ & (0) & \ & \ddots &\ \\ \ & \ & \ &\
& -1
\end{array}\right]\, :\ell^2_n\longrightarrow\ell^2_n\,.
$$
Since $U^{2}$ is the identity on $\ell^{2}_n$, we have
$F(U)^{2}=I_{\Lambda_n}$, hence $\widehat{U}^{2}$ is the identity
operator on $B(\Lambda_n)$. We set
$$
\Delta\,=\,\frac{Id +\widehat{U}}{2}\,.
$$
Then $\Delta$ is a projection. Moreover $\widehat{U}$ is a
complete contraction on $S^p(\Lambda_n)$ by Lemma \ref{2Mult},
hence $\Delta\colon S^p(\Lambda_n)\to S^p(\Lambda_n)$ is a
complete contraction as well. We let
$$
\Delta^{\oplus n}\colon \mathop{\oplus}\limits^{p}_{1\leq k\leq n}
S^p(\Lambda_n)\otimes a_k \longrightarrow
\mathop{\oplus}\limits^{p}_{1\leq k\leq n} S^p(\Lambda_n)\otimes
a_k
$$
be the amplification of $\Delta$ taking $(z_1\otimes
a_1,\ldots,z_n\otimes a_n)$ to $(\Delta(z_1)\otimes
a_1,\ldots,\Delta(z_n)\otimes a_n)$ for any $z_1,\ldots, z_n$ in
$S^p(\Lambda_n)$. Clearly, $\Delta^{\oplus n}$ also is a
completely contractive projection.

It follows from Lemma \ref{4Quant} that for any $1\leq k\leq n$,
$\Delta$ maps $H_{n,k}$ into itself, and that
\begin{align*}
\Delta(\varphi_k(s))\,= &\,\varphi_k\bigl(\langle s,e_1\rangle e_1
+
\langle s,e_2\rangle e_2\bigr)\\
& =\,\langle s,e_1\rangle  c_{n,1,k} + \langle s,e_2\rangle
c_{n,2,k},\qquad s\in\ell_n^{2}.
\end{align*}
Thus $\Delta^{\oplus n}$ maps $E$ into itself, and
$$
\Delta^{\oplus n}(E)\,=\,\Bigl\{\bigl( (s_1c_{n,1,1} +s_2
c_{n,2,1})\otimes a_1,\ldots, (s_1c_{n,1,n} +s_2 c_{n,2,n})\otimes
a_n\bigr)\, :\, s_1,s_2\in\Cdb\Bigr\}
$$
is completely $1$-complemented in $E$. Therefore, the
$2$-dimensional Hilbert space $\Delta^{\oplus n}(E)$ is
$[2]$-$1$-complemented in the $p$-direct sum of the
$S^p(\Lambda_n)\otimes a_k$'s.

Let $1\leq k\leq n$. Given any $s_1,s_2\in\Cdb$, let us look at
the matrix $M_k(s_1,s_2)$ of the operator $s_1c_{n,1,k} +s_2
c_{n,2,k}$ in the canonical bases $\{e_B\, :\, \vert B\vert
=k-1\}$ and $\{e_A\, :\, \vert A\vert =k\}$ of $\Lambda_{n,k-1}$
and $\Lambda_{n,k}$ (see Section 2). We call $m_{A,B}$ the entries
of this matrix. If $1\notin B$ and $2\notin B$, then
$m_{A_1,B}=s_1$ and $m_{A_2,B}=s_2$, where $A_1=B\cup\{1\}$ and
$A_2=B\cup\{2\}$. All other entries in the column indexed by $B$
are equal to $0$. This case occurs $\binom{n-2}{k-1}$ times.
Otherwise, that is  if $1\in B$ or $2\in B$, then the column
indexed by $B$ has at most one non zero entry. Likewise if
$\{1,2\}\subset A$, then $m_{A,B_1}=s_1$ and $m_{A,B_2}= -s_2$,
where $B_1=A\setminus\{1\}$ and $B_2=A\setminus\{2\}$. All other
entries in the row indexed by $A$ are equal to $0$. This case
occurs $\binom{n-2}{k-2}$ times. Otherwise, that is if $1\notin A$
or $2\notin A$, then the row indexed by $A$ has at most one non
zero entry. Furthermore the submatrices
$$
\left[\begin{array}{c} s_1\\s_2
\end{array}\right]\qquad\hbox{and}\qquad
\left[\begin{array}{cc} s_1 & - s_2
\end{array}\right]
$$
appearing in $M_k(s_1,s_2)$ are `orthogonal' to each other. Namely
if $m_{A_1,B}=s_1$, $m_{A_2,B}=s_2$, $m_{A,B_1}=s_1$ and
$m_{A,B_2}= -s_2$, then $A$ is both different from $A_1$ and $A_2$
and $B$ is both different from $B_1$ and $B_2$. Consequently for
an appropriate ordering of the canonical bases of
$\Lambda_{n,k-1}$ and $\Lambda_{n,k}$, we have a block diagonal
representation
$$
M_k(s_1,s_2)\,=\, \left[\begin{array}{ccc} \left[\begin{array}{c}
s_1\\s_2
\end{array}\right]\otimes I_{k,1} & \ & \ \\
\ & \left[\begin{array}{cc} s_1 & - s_2
\end{array}\right]\otimes I_{k,2} &\ \\ \ &\ & 0_{k,3}
\end{array}\right],
$$
where $I_{k,1}$ is the unit of the square matrices of size
$\binom{n-2}{k-1}$, $I_{k,2}$ is the unit of the square matrices
of size $\binom{n-2}{k-2}$, and $0_{k,3}$ is the zero rectangular
matrix of size $\binom{n-2}{k}\times \binom{n-2}{k-3}$. We deduce
that there exist two operators $b_1\in S^p(K_1)$ and $b_2\in
S^p(K_2)$ (defined on large enough Hilbert spaces $K_1,K_2$) such
that
$$
\Delta^{\oplus n}(E)\sim Y:\,=\,\left\{
\left(\left[\begin{array}{c} s_1\\s_2
\end{array}\right]\otimes b_1,\left[\begin{array}{cc} s_1 & - s_2
\end{array}\right]\otimes b_2\right)\, :\, s_1,s_2\in\Cdb \right\}
$$
and
\begin{equation}\label{4Norm}
\norm{b_1}^p=\sum_{k=1}^{n}
\tbinom{n-2}{k-1}\norm{a_k}_p^p\qquad\hbox{and}\qquad
\norm{b_2}^p=\sum_{k=1}^{n} \tbinom{n-2}{k-2}\norm{a_k}_p^p.
\end{equation}
By Lemma \ref{2Equiv2}, $Y$ is $[2]$-$1$-complemented in
$S^p(\ell^2_2(K_1),K_1)\mathop{\oplus}\limits^p S^p(K_2,\ell^2_2(
K_2) )$, which implies by Proposition \ref{3Rect} that $b_1=0$ or
$b_2=0$. For any $1\leq k\leq n-1$, we have
$\binom{n-2}{k-1}\not=0$. Hence in the case when $b_1=0$, we have
$a_1=\cdots=a_{n-1}=0$ by (\ref{4Norm}). Likewise in the case when
$b_2=0$, we have $a_2=\cdots=a_{n}=0$.
\end{proof}

\medskip
\section{Spinorial subspaces}
In this section we will prove that spinorial spaces (as defined in
Definition \ref{2Spin}) of dimension $\geq 5$ cannot be
$[2]$-$1$-complemented. As an intermediate step of independent
interest, we will consider a variant of these spaces, using the
Fermions and Clifford algebras. The necessary background on these
topics can be found in \cite{BR,PR}.

We need a few simple facts about noncommutative $L^p$-spaces and
their completely bounded maps. The following definitions extend
those given in Section 2 for the Schatten spaces. If $M$ is any
semifinite von Neumann algebra equipped with a normal semifinite
faithful trace $\varphi$, and $1\leq p<\infty$, we let
$L^p(M,\varphi)$ (or simply $L^p(M)$) denote the associated
noncommutative $L^p$-space. Recall that if we let
$$
\norm{x}_p=\bigl(\varphi(\vert
x\vert^p)\bigr)^{\frac{1}{p}},\qquad x\in M,
$$
then $L^p(M,\varphi)$ is the completion of the space $\{x\in M\,
:\, \norm{x}_p<\infty\}$ equipped with $\norm{\ }_p$. See e.g.
\cite{PX} for information on these spaces. If $(M_1,\varphi_1)$
and $(M_2,\varphi_2)$ are two semifinite von Neumann algebra, let
$M_1\overline{\otimes}M_2$ denote the von Neumann algebra tensor
product and let $\varphi_1\overline{\otimes}\varphi_2$ denote the
associated semifinite faithful trace. Then
$$
L^p(M_1,\varphi_1)\otimes L^p(M_2,\varphi_2)\subset
L^p(M_1\overline{\otimes}M_2,
\varphi_1\overline{\otimes}\varphi_2)
$$
is a dense subspace. For any closed subspace $X\subset
L^p(M_2,\varphi_2)$, we denote by
$$
L^p(M_1,\varphi_1)\mathop{\otimes}\limits^p X\subset
L^p(M_1\overline{\otimes}M_2,\varphi_1\overline{\otimes}\varphi_2)
$$
the closure of $L^p(M_1,\varphi_1)\otimes X$ in
$L^p(M_1\overline{\otimes}M_2,\varphi_1\overline{\otimes}\varphi_2)$.

If $u\colon X\to Y$ is any bounded linear map between two
subspaces of noncommutative $L^p$-spaces, we define $\nnorm{u}$
and $\cbnorm{u}$ by (\ref{2n}) and (\ref{2cb}). Then the
terminology introduced in the second paragraph of Section 2 also
extends to this context. Again we refer to \cite{P1} for the
connections with operator space theory and further information. We
will use repeatedly the following well-known easy fact.

\begin{lemma}\label{5Ext}
Assume that $(M_1,\varphi_1)$ and $(M_2,\varphi_2)$ are finite von
Neumann algebras and let $\pi\colon M_1\to M_2$ be a one-to-one
$*$-representation. Let $\delta>0$ be a constant and assume that
$\varphi_2(\pi(x))=\delta\varphi_1(x)$ for any $x\in M_{1}$. Then
for any $1\leq p<\infty$, $\delta^{-p}\pi$ (uniquely) extends to a
complete isometry from $L^p(M_1)$ into $L^p(M_2)$.
\end{lemma}

For any Hilbert space $\H$ we write $tr$ for the usual trace on
$B(\H)$. In the sequel, the semifinite von Neumann algebras we
will meet will be either finite dimensional ones or the Schatten
spaces $S^p(\H)=L^p(B(\H),tr)$.

\bigskip
Let $N\geq 1$ be an integer. For convenience we write $1$ for the
identity operator on $B(\Lambda_N)$. Recall that the Fermions
$\omega_j=c_{N,j}+c_{N,j}^{*}\in B(\Lambda_n)$ are selfadjoint
unitaries which anti-commute, that is,
$$
\forall\,j,\quad \omega_j^2=1\ \hbox{ and }\
\omega_j^*=\omega_j;\qquad \forall\,j\not= j',\quad
\omega_j\omega_{j'}=-\omega_{j'}\omega_{j}.
$$
These properties will be used throughout without any further
comments. As an immediate consequence, we have
\begin{equation}\label{5Square}
(\omega_1\cdots\omega_{N-1}\omega_N)^{2}\,=\,(-1)^{\frac{N(N-1)}{2}}\,.
\end{equation}
The Clifford algebra with $N$ generators is the $C^{*}$-algebra
$$
\C_N\,=\,
C^{*}\bigl\langle\omega_1,\ldots,\omega_N\bigr\rangle\subset
B(\Lambda_N)
$$
generated by the first $N$ Fermions. The dimension of $\C_N$ is
equal to $2^{N}$. More precisely, for any $A\in\P_{N}$ (the set of
all subsets of $\{1,\ldots,N\}$), set
$\omega_A=\omega_{i_1}\cdots\omega_{i_\ell}$, when
$A=\{i_1,\ldots,i_{\ell}\}$ and $i_1<\cdots<i_\ell$. By
convention, $\omega_{\emptyset}=1$. Then
$\{\omega_A\,:\,A\in\P_N\}$ is a basis of $\C_N$.

Recall that $\Omega\in\Lambda_N$ denotes the vacuum vector. The
functional $Tr\colon\C_N\to\Cdb$ defined by
$$
Tr(x)=\langle x(\Omega),\Omega\rangle,\qquad x\in\C_N,
$$
is a normalized trace on $\C_N$. For $1\leq p<\infty$, we let
$L^p(\C_N)$ denote the associated noncommutative $L^p$-space. In
the sequel, by an orthogonal projection $\C_N\to\C_N$, we will
simply mean a projection which is orthogonal on the Hilbert space
$L^2(\C_N)$. It turns out (easy to check) that
$\{\omega_A\,:\,A\in\P_N\}$ is an orthonormal basis of
$L^{2}(\C_N)$.

We will focus on the operator space
$$
E_N={\rm Span}\{1,\omega_1,\ldots,\omega_N\}\subset \C_N.
$$

\begin{lemma}\label{5OrthoE}
Let $P\colon\C_N\to\C_N$ be the orthogonal projection onto $E_N$.
Then
$$
\bignorm{P\colon L^p(\C_N)\longrightarrow L^p(\C_N)}=1.
$$
for any $1\leq p\leq\infty$.
\end{lemma}

\begin{proof}
We will first show that $P$ is positive. Let $x\in \C_N$ with
$x\geq 0$. In particular $x$ is selfadjoint. Since $E_N$ and
$E_N^{\perp}$ are both selfadjoint subspaces of $\C_N$, this
implies that $P(x)$ is selfadjoint as well. Thus there exist real
numbers $\alpha,\beta_1,\ldots,\beta_N$ such that
$$
P(x)=\alpha 1\, +\,\sum_{j=1}^N \beta_j\omega_j\,.
$$
Let
$$
\beta =\Bigl(\sum_{j=1}^N
\beta_j^2\Bigr)^{\frac{1}{2}}\qquad\hbox{and}\qquad
\omega=\beta^{-1}\, \sum_{j=1}^N \beta_j\omega_j,
$$
if $\beta\not= 0$. Since the $\beta_j$'s are real, we have
$\omega^*=\omega$ and the anticommutation relations yield
$$
\omega^2 = \beta^{-2}\sum_{j,j'}\beta_j\beta_{j'}
\omega_{j}\omega_{j'}\,=\,\beta^{-2}\sum_j \beta_j^2 \, =1.
$$
Thus $\omega$ is a selfadjoint unitary. Let $q_+=2^{-1}(1+\omega)$
and $q_-=2^{-1}(1-\omega)$. Then $q_+$ and $q_-$ are orthogonal
projections with sum $q_++q_-=1$ and
$$
P(x)=(\alpha +\beta)q_+\, +\,(\alpha-\beta)q_-.
$$
Since $Tr(\omega_j)=0$ for any $j$, we have $Tr(\omega)=0$. Hence
$Tr(q_+)=1/2$. Consequently,
$$
Tr(xq_+)=Tr(xP(q_+))=Tr(P(x)q_+)=\frac{\alpha+\beta}{2}\,.
$$
Since $x\geq 0$, we have $Tr(xq_+)=Tr(q_+xq_+)\geq 0$ hence we
have proved that $\alpha+\beta\geq 0$. Likewise, $\alpha-\beta\geq
0$ and we deduce that $P(x)\geq 0$. The argument works as well if
$\beta=0$.

We have shown that the map $P\colon\C_N\to C_N$ is positive. Since
it is unital, it is a contraction (see e.g. \cite[Cor. 2.9]{Pa}).
Since $P$ is selfadjoint, we obtain for free that $P\colon
L^1(\C_N)\to L^1(\C_N)$ also is a contraction. We deduce by
interpolation that $P\colon L^p(\C_N)\to L^p(\C_N)$ is a
contraction for any $1<p<\infty$. Indeed,
$L^p(\C_N)=[\C_N,L^1(\C_N)]_{\frac{1}{p}}$, where $[\ ,\ ]_\theta$
denotes the complex interpolation method, see e.g. \cite{PX} for
details.
\end{proof}

We shall now discuss several facts depending on the parity of $N$.
It is well-known that for any integer $n\geq 1$,
\begin{equation}\label{5*}
\C_{2n}\simeq M_{2^n}\qquad *\hbox{-isomorphically}.
\end{equation}
Moreover this identification induces an isometric identification
$$
L^{p}(\C_{2n})\simeq S^p_{2^n}
$$
for any $p\geq 1$. Indeed if $\pi\colon M_{2^n} \to \C_{2n}$ is
the canonical $*$-isomorphism, then we have
$Tr(\pi(x))=2^{-n}tr(x)$ for any $x\in M_{2^n}$. This implies by
Lemma \ref{5Ext} that for any $p\geq 1$,
$$
2^{\frac{n}{p}}\,\pi\colon S^p_{2^n} \longrightarrow
L^p(\C_{2n})\quad\hbox{ is a complete isometry.}
$$

We now consider the odd case. For any $n\geq 1$, we set
\begin{equation}\label{5rho0}
\rho_n=\,\frac{1}{2}\bigl(1+i^n
\omega_1\cdots\omega_{2n}\omega_{2n+1}\bigr) \in \C_{2n+1}.
\end{equation}
From (\ref{5Square}) we have $(i^n
\omega_1\cdots\omega_{2n}\omega_{2n+1})^{2}=1$, hence $\rho_n$ is
a (non trivial) selfadjoint projection. Moreover the
anti-commutation relations imply that
$$
(\omega_1\cdots\omega_{2n}\omega_{2n+1})\omega_j=\omega_j
(\omega_1\cdots\omega_{2n}\omega_{2n+1})
$$
for any $j=1,\ldots, 2n+1$. Thus
$\omega_1\cdots\omega_{2n}\omega_{2n+1}$ lies in the center of
$\C_{2n+1}$ and $\rho_n$ is therefore central. This induces a
direct sum decomposition
\begin{equation}\label{5Decomp1}
\C_{2n+1}\,=\,\rho_n\C_{2n+1}
\mathop{\oplus}\limits^{\infty}(1-\rho_n)\C_{2n+1}.
\end{equation}
Regarding $\C_{2n}$ as a subalgebra of $\C_{2n+1}$ in the obvious
way, we have
\begin{equation}\label{5rho}
\rho_n\C_{2n}=\rho_n\C_{2n+1}.
\end{equation}
Indeed note that $\C_{2n+1}$ is spanned by $\C_{2n}$ and the set
$\{\omega_A\omega_{2n+1}\, :\, A\in\P_{2n}\}$. Thus to get this
equality, it suffices to check that for any $A\in\P_{2n}$, we have
$\rho_n\omega_A\omega_{2n+1}\in \rho_n\C_{2n}$. We have
$$
2\rho_n\omega_A\omega_{2n+1}=\omega_A\omega_{2n+1} +
i^n\omega_1\cdots \omega_{2n+1}\omega_A\omega_{2n+1}=
\omega_A\omega_{2n+1} + i^n(-1)^{\vert A\vert}\omega_1\cdots
\omega_{2n}\omega_A.
$$
Let $y=i^n(-1)^{\vert A\vert}\omega_1\cdots \omega_{2n}\omega_A$.
Then $y\in\C_{2n}$ and
\begin{align*}
i^n\omega_1\cdots \omega_{2n+1}y & =(-1)^n(-1)^{\vert A\vert}
\omega_1\cdots \omega_{2n}\omega_{2n+1}\omega_1\cdots
\omega_{2n}\omega_A\\ & =(-1)^n(\omega_1\cdots
\omega_{2n})^{2}\omega_A\omega_{2n+1}\\ & =\omega_A\omega_{2n+1},
\end{align*}
by (\ref{5Square}). Hence $\rho_n\omega_A\omega_{2n+1}=\rho_n y$,
which proves the result.

Since $\C_{2n}$ is simple, the $*$-representation
$$
\pi_0\colon\C_{2n}\longrightarrow \C_{2n+1},\qquad x\mapsto \rho_n
x,
$$
is one-to-one. The equality we just proved shows that its range is
equal to $\rho_n\C_{2n+1}$. Likewise, the  $*$-representation
$$
\pi_1\colon\C_{2n}\longrightarrow\C_{2n+1},\qquad x\mapsto
(1-\rho_n)x,
$$
is a $*$-isomorphism from $\C_{2n}$ onto $(1-\rho_n)\C_{2n+1}$.
Hence the decomposition (\ref{5Decomp1}) induces $*$-isomorphisms
\begin{equation}\label{5Decomp2}
\C_{2n+1}\simeq\C_{2n}\mathop{\oplus}\limits^{\infty}\C_{2n}
\simeq M_{2^n}\mathop{\oplus}\limits^{\infty}M_{2^n}.
\end{equation}
We observe that $Tr(\pi_0(x))=Tr(\pi_1(x)) =\frac{1}{2}Tr(x)$ for
any $x\in\C_{2n}$. By Lemma \ref{5Ext}, this implies that for any
$p\geq 1$,
\begin{equation}\label{5Iso0}
2^{\frac{1}{p}} \pi_0,\ 2^{\frac{1}{p}} \pi_1\colon
L^p(\C_{2n})\longrightarrow L^p(\C_{2n+1})\qquad\hbox{are complete
isometries.}
\end{equation}
This yields canonical isometric identifications
$$
L^p(\C_{2n+1})\simeq
L^p(\C_{2n})\mathop{\oplus}\limits^{p}L^p(\C_{2n}) \simeq
S^p_{2^n}\mathop{\oplus}\limits^{p} S^p_{2^n}.
$$

We introduce
$$
F_n={\rm Span}\{1,\omega_1,\ldots,\omega_{2n},
\omega_1\cdots\omega_{2n}\}\subset \C_{2n}
$$
This operator space is closely related to $E_{2n+1}$. Indeed owing
to the calculation we made to prove (\ref{5rho}), we have
\begin{equation}\label{5rhon1}
\rho_n\omega_{2n+1}=\rho_n(i^n\omega_1\cdots\omega_{2n}).
\end{equation}
Hence $\pi_0(F_n)= \rho_n E_{2n+1}$. Likewise, we have
\begin{equation}\label{5rhon2}
(1-\rho_n)\omega_{2n+1}=-(1-\rho_n)(i^n\omega_1\cdots\omega_{2n}),
\end{equation}
and $\pi_1(F_n)=(1-\rho_n)E_{2n+1}$. Arguing as in Lemma
\ref{5OrthoE}, we have the following.

\begin{lemma}\label{5OrthoF}
Let $Q\colon\C_{2n}\to\C_{2n}$ be the orthogonal projection onto
$F_n$. Then
$$
\bignorm{Q\colon L^p(\C_{2n})\longrightarrow L^p(\C_{2n})}=1
$$
for any $1\leq p\leq\infty$.
\end{lemma}

For any $p\geq 1$, we let $E_N^p$ denote the space $E_N$ regarded
as a subspace of $L^p(\C_{N})$. Likewise for any $n\geq 1$ we let
$F_n^p$ denote the space $F_n$ regarded as a subspace of
$L^p(\C_{2n})$. We define a `transpose map'
$$
\tau\colon F_n\longrightarrow F_n
$$
by letting $\tau(1)=1$, $\tau(\omega_j)=\omega_j$ for any
$j=1,\ldots, 2n$, and $\tau(\omega_1\cdots\omega_{2n})
=-\omega_1\cdots\omega_{2n}$. The following fact will be used
later on in this section.

\begin{lemma}\label{5tau1}
Consider
$$
\theta= \pi_1\tau\pi_0^{-1}\colon \rho_nE_{2n+1}\longrightarrow
(1-\rho_n)E_{2n+1}.
$$
Then $\theta(\rho_n)=1-\rho_n$ and $\theta(\rho_n\omega_j)=
(1-\rho_n)\omega_j$ for any $j=1,\ldots, 2n+1$.
\end{lemma}

\begin{proof}
Only the relation $\theta(\rho_n\omega_{2n+1})=
(1-\rho_n)\omega_{2n+1}$ needs a proof. This follows from
(\ref{5rhon1}) and (\ref{5rhon2}).
\end{proof}

\begin{remark}\label{5tau2}\

(1) In the case when $n=1$, we have $F_1=\C_2$. Consider the
so-called Pauli matrices defined by
$$
a =\left[\begin{array}{cc} 1 & 0 \\ 0 & -1
\end{array}\right],\quad b=\left[\begin{array}{cc} 0 & 1 \\ 1 & 0
\end{array}\right]\quad\hbox{ and }\quad
c=\left[\begin{array}{cc} 0 & 1 \\ -1 & 0
\end{array}\right].
$$
Then the $*$-isomorphism $\pi\colon M_2\to \C_2$ yielding
(\ref{5*}) in the case $n=1$ is defined by $\pi(1)=1$,
$\pi(a)=\omega_1$, $\pi(b)=\omega_2$ and
$\pi(c)=\omega_1\omega_2$. Thus  $\tau\colon F_1\to F_1$
corresponds to the classical transpose map of $M_2$.


\smallskip
(2) Let $\tau'\colon F_n\to F_n$ be defined by letting
$\tau'(1)=1$, $\tau'(\omega_j)=-\omega_j$ for any $j=1,\ldots,2n$,
and $\tau'(\omega_1\cdots\omega_{2n}) = -
\omega_1\cdots\omega_{2n}$. Let $Q$ be the projection introduced
in Lemma \ref{5OrthoF}. Adapting the argument in the proof of
Lemma \ref{5OrthoE}, one obtains that $\tau'Q\colon\C_{2n}\to
\C_{2n}$ is a positive, unital, selfadjoint operator, and hence
that $\tau'Q\colon L^p(\C_{2n})\to L^p(\C_{2n})$ is a contraction
for any $p\geq 1$. By restriction, we deduce that $\tau'\colon
F_n^p\to F_n^p$ is a contraction. Since $\tau'$ is an involution,
this is actually an isometry. Let $\pi\colon \C_{2n}\to\C_{2n}$ be
the $*$-isomorphism taking $\omega_j$ to $-\omega_j$ for any
$j=1,\ldots, 2n$. Then $\pi\colon L^p(\C_{2n})\to L^p(\C_{2n})$ is
a complete isometry for any $p$ (see Lemma \ref{5Ext}) and
$\pi\tau'=\tau$. We deduce that for any $p\geq 1$, $\tau$ is an
isometry on $F^p_n$.

However in general, $\tau\colon F_n^p\to F_n^p$ is not a complete
isometry. Indeed by (1) above and Proposition \ref{3Sym},
$\tau\colon F_1^p\to F_1^p$ is not completely contractive unless
$p=2$. The question whether $\tau\colon F_n^p\to F_n^p$ is a
complete isometry is a key issue for our understanding of the
$F_n^p$'s as operator spaces. This will be discussed in details in
Section 7 below.

\smallskip
(3) Let $\sigma\colon\C_{2n}\to \C_{2n}$ be the (necessarily
unique) anti-$*$-isomorphism such that $\sigma(\omega_j)=\omega_j$
for any $j=1,\ldots,2n$. According to (\ref{5Square}), we have
$\sigma(\omega_1\cdots\omega_{2n})=(-1)^{n}\omega_1\cdots\omega_{2n}$.
Thus the restriction $\sigma_{\vert F_n}$ is equal to $\tau$ if
$n$ is odd and is equal to $I_{F_n}$ if $n$ is even.

Let $F_n^{\rm op}$ be the space $F_n$ equipped with the opposite
operator space structure (see e.g. \cite[Section 2.10]{P2}). The
mapping $\sigma$ is a $*$-homomorphism from $\C_{2n}$ into the
opposite $C^*$-algebra $\C_{2n}^{\rm op}$, and $F_n^{\rm
op}\subset \C_{2n}^{\rm op}$ completely isometrically. Hence by
the above paragraph, $\tau\colon F_n\to F_n^{\rm op}$ is a
complete isometry if $n$ is odd whereas $I_{F_n} \colon F_n\to
F_n^{\rm op}$ is a complete isometry  if $n$ is even.
\end{remark}

\bigskip

\begin{proposition}\label{5OrthoEbis}
Let $N\geq 2$ be an integer and let $P\colon\C_N\to\C_N$ be the
orthogonal projection onto $E_N$. Then for any $1\leq p\not =2
\leq\infty$, we have
$$
\norm{P\colon L^p(\C_N)\longrightarrow L^p(\C_N)}_2>1.
$$
\end{proposition}

\begin{proof}
Let $\A={\rm Span}\{1,\omega_1,\omega_2,\omega_1\omega_2\}\subset
\C_N$ and let $\A^p$ be this space regarded as a subspace of
$L^p(\C_N)$. Then $P$ maps $\A$ onto its subspace ${\rm
Span}\{1,\omega_1,\omega_2\}$. Under the identification given in
Remark \ref{5tau2} (1), the latter space coincides with the space
symmetric $2\times 2$ matrices and we deduce that
$$
\bignorm{P_{\vert\footnotesize{\A}}\colon
\A^p\longrightarrow\A^p}_2\, =\, \bignorm{P_s\colon
S_2^p\longrightarrow S_2^p}_2,
$$
where $P_s$ denotes the canonical projection onto $\S_2^p$. The
result therefore follows from  Proposition \ref{3Sym}.
\end{proof}

\bigskip
For any operators $a_1\in S^p(H_1)$ and $a_2\in S^p(H_2)$, with
$(a_1,a_2)\not= (0,0)$, let us consider the following analog of
(\ref{2Spinform}):
\begin{equation}\label{5G}
G=\bigl\{(x\otimes a_1, \tau(x)\otimes a_2)\, :\, x\in
F_n^p\bigr\}\subset (L^p(\C_{2n}) \mathop{\otimes}\limits^p
S^p(H_1)) \mathop{\oplus}\limits^p (L^p(\C_{2n})
\mathop{\otimes}\limits^p S^p(H_2)).
\end{equation}

\begin{proposition}\label{5G1}
Assume that $1\leq p\not= 2<\infty$ and that $n\geq 2$. Then the
above space $G$ is $1$-complemented but is not
$[2]$-1-complemented.
\end{proposition}

\begin{proof}
Let us assume that $\norm{a_1}^p+\norm{a_2}^p=1$ and let
$t=\norm{a_1}^p$. Recall the projection $Q$ from Lemma
\ref{5OrthoF}. We let
$$
R\colon (L^p(\C_{2n})\otimes a_1)\mathop{\otimes}\limits^p
(L^p(\C_{2n})\otimes a_2)\longrightarrow (L^p(\C_{2n})\otimes
a_1)\mathop{\otimes}\limits^p (L^p(\C_{2n})\otimes a_2)
$$
be the linear mapping defined by
$$
R(z_1\otimes a_1, z_2\otimes a_2)\,=\,\bigl((tQ(z_1) +(1-t)\tau
Q(z_2))\otimes a_1,(t\tau Q(z_1) +(1-t)Q(z_2))\otimes a_2\bigr),
$$
for any $z_1,z_2\in L^p(\C_{2n})$. This definition is an analog of
(\ref{2Proj}). Arguing as in Section 2 and using Lemma
\ref{5OrthoF}, we obtain that $R$ is a contractive projection.

To show that $G$ is not $[2]$-1-complemented, it suffices by
Remark \ref{2Uniqueness} to show that the above mapping $R$ is not
$[2]$-contractive. The proof is similar to the one of Proposition
\ref{5OrthoEbis}. Again we consider $\A={\rm
Span}\{1,\omega_1,\omega_2,\omega_1\omega_2\}$ and we note that
since $n\geq 2$, $\tau Q(z)=Q(z)$ for any $z\in\A$. Thus
$$
R(z\otimes a_1, z\otimes a_2)\,=\,(Q(z)\otimes a_1,Q(z)\otimes
a_2),\qquad z\in\A.
$$
Consequently,
$$
\norm{R}_2\geq\bignorm{Q\colon \A^p\longrightarrow \A^p }_2 =
\bignorm{P_s\colon S^p_2\longrightarrow S^p_2}_2,
$$
and the latter norm is $>1$ by Proposition \ref{3Sym}.
\end{proof}

\begin{remark}\label{5G2}
\

(1) As a special case $(a_1=a, a_2=0)$, we obtain that for any
$n\geq 2$ and any non zero $a\in S^p(H)$, the space
$$
F_n\otimes a\subset L^p(\C_{2n}) \mathop{\otimes}\limits^p S^p(H)
$$
is not $[2]$-1-complemented for $p\not=2$. Equivalently,
$$
\norm{Q\colon L^p(\C_N)\longrightarrow L^p(\C_N)}_2>1
$$
whenever $p\not=2$.

\smallskip
(2) It follows from Proposition \ref{5OrthoEbis} that for any
$N\geq 2$ and non zero $a\in S^p(H)$, the space
$$
E_N\otimes a\subset L^p(\C_{N}) \mathop{\otimes}\limits^p S^p(H)
$$
is not $[2]$-1-complemented for $p\not=2$.
\end{remark}

\bigskip
Our next goal is to prove Theorem \ref{5Spin} below. We need more
information on spin systems. We noticed in Section 2 that for any
$N\geq 1$, the Fermions $(\omega_1,\ldots,\omega_{N})$ form a spin
system. Also it follows from (\ref{5Square}) that for any $n\geq
1$, the $(2n+1)$-tuple $(\omega_1,\ldots,\omega_{2n}, i^n
\omega_1\cdots\omega_{2n})$ is a spin system. The next lemma shows
that these are essentially the only examples.

\begin{lemma}\label{5UniqueSpin} Let $n\geq 1$ be any integer.
\begin{itemize}
\item [(1)] Let $(s_1,\ldots,s_{2n})$ be a spin system with an
even cardinal. There is a (necessarily unique) $*$-isomorphism
$$
\pi\colon \C_{2n}\longrightarrow C^*\langle s_1,\ldots,
s_{2n}\rangle
$$
such that $\pi(\omega_j)= s_j$ for any $j=1,\ldots, 2n$. \item
[(2)] Let $(s_1,\ldots,s_{2n+1})$ be a spin system with an odd
cardinal and let
$$
q=\frac{1}{2}\bigl(1+i^n s_1\cdots s_{2n}s_{2n+1}\bigr).
$$

If $q\notin\{0,1\}$, then there is a (necessarily unique)
$*$-isomorphism
$$
\pi\colon \C_{2n+1}\longrightarrow C^*\langle s_1,\ldots,
s_{2n+1}\rangle
$$
such that $\pi(\omega_j)=s_j$ for any $j=1,\ldots, 2n+1$. In this
case, we have $\pi(\rho_n)=q$.
\end{itemize}
\end{lemma}

\bigskip
Note that in (2) above, $q=0$ if and only if $s_{2n+1}= - i^n
s_1\cdots s_{2n}$. In this case (1) ensures that there is a
$*$-isomorphism $\pi\colon \C_{2n}\to C^*\langle s_1,\ldots,
s_{2n}\rangle$ such that $\pi(\omega_j)=s_j$ for any $j\leq 2n$
and $\pi(i^n\omega_1\cdots\omega_{2n})=s_{2n+1}$. A similar
comment applies when $q=1$.

\bigskip
For simplicity, we now let $c_j$ (instead of $c_{n,j}$) denote the
creation operators on $\Lambda_n$ and recall that we defined
$\omega_j = c_j + c_j^*$ for any $j=1,\ldots, n$. Next we let
\begin{equation}\label{5Om-}
\omega_{-j} = \frac{c_j -c_j^*}{i}\,\qquad j=1,\ldots, n.
\end{equation}
It is well-known (and easy to check) that the $2n$-tuple
$(\omega_1,\ldots,\omega_n, \omega_{-1},\ldots, \omega_{-n})$ is a
spin system.

\begin{proof}[{\it Proof of Lemma \ref{5UniqueSpin}.}]
Let $(s_1,\ldots,s_{2n})$ be an arbitrary spin system and set
$$
v_j=\frac{s_j + i s_{n+j}}{2}\,,\qquad j=1,\ldots,n.
$$
These operators satisfy the so-called canonical anti-commutation
relations (CAR), that is,
$$
v_iv_j^* + v_j^*v_i=\delta_{i,j}\qquad\hbox{and}\qquad  v_iv_j +
v_jv_i=0
$$
for any $1\leq i,j\leq n$. The creation operators $c_1,\ldots,
c_n$ satisfy the CAR as well hence according to e.g. \cite[p.
15]{BR}, there is a $*$-isomorphism $\pi\colon C^*\langle
c_1,\ldots, c_n\rangle\to C^*\langle v_1,\ldots, v_n\rangle$ such
that $\pi(c_j)=v_j$ for all $j$. Equivalently,
$$
\pi\colon C^*\langle \omega_1,\ldots,
\omega_n,\omega_{-1},\ldots,\omega_{-n}\rangle\longrightarrow
C^*\langle s_1,\ldots, s_{2n}\rangle
$$
is a $*$-isomorphism which satisfies $\pi(\omega_j)= s_j$ and
$\pi(\omega_{-j})=s_{n+j}$ for any $j=1,\ldots, n$. The assertion
(1) follows at once.

Now let $(s_1,\ldots,s_{2n}, s_{2n+1})$ be a spin system with an
odd cardinal, and let $\omega_1,\ldots,\omega_{2n},\omega_{2n+1}$
be the usual Fermions. Suppose that $q\notin\{0,1\}$. Then we can
mimic what we did before with Fermions and we obtain that $q$  is
a central projection of $C^*\langle s_1,\ldots, s_{2n+1}\rangle$,
and that we have
\begin{align*}
C^*\langle s_1,\ldots, s_{2n+1}\rangle & = q C^*\langle
s_1,\ldots, s_{2n+1}\rangle\mathop{\oplus}\limits^\infty (1-q)
C^*\langle s_1,\ldots, s_{2n+1}\rangle\\
& \simeq C^*\langle s_1,\ldots,
s_{2n}\rangle\mathop{\oplus}\limits^\infty C^*\langle s_1,\ldots,
s_{2n}\rangle
\end{align*}
Then using the $*$-isomorphism $\C_{2n}\to C^*\langle s_1,\ldots,
s_{2n}\rangle$ given by (1), we deduce the desired $*$-isomorphism
from $\C_{2n+1}$ onto $C^*\langle s_1,\ldots, s_{2n+1}\rangle$.
\end{proof}

As in Section 2, we let $P_n\colon \Lambda_n\to\Lambda_n$ be the
orthogonal projection onto the space generated by tensor products
of even rank.

\begin{lemma}\label{5Pn} With the notation introduced before Lemma
\ref{5UniqueSpin}, we have
$$
P_n=\frac{1}{2}\bigl(1+i^n\omega_n\omega_{-n}\cdots\omega_{1}\omega_{-1}\bigr).
$$
\end{lemma}

\begin{proof}
Let $A\subset\{1,\ldots,n\}$ and recall (\ref{2EA}). For any
$j=1,\ldots, n$, we have
$$
\omega_j\omega_{-j} =-i(c_j+c_j^*)(c_j-c_j^*)
=i(c_jc_j^*-c_j^*c_j).
$$
Hence $\omega_j\omega_{-j}(e_A)=ie_A$ if $j\in A$, and
$\omega_j\omega_{-j}(e_A)=-ie_A$ if $j\notin A$. Consequently
$$
\omega_{n}\omega_{-n}\cdots\omega_{1}\omega_{-1}(e_A) =
(-1)^{n-\vert A\vert} i^n e_A,
$$
which implies the result.
\end{proof}

\begin{theorem}\label{5Spin} Let $X\subset S^p(\H,\K)$ with
${\rm dim}(X)\geq 5$ and $1\leq p\not=2<\infty$. If $X$ is a
spinorial space, then $X$ is not $[2]$-$1$-complemenented.
\end{theorem}

\begin{proof}
We first consider spinorial spaces with an even dimension. Let
$n\geq 3$ be an integer. By Lemma \ref{2Equiv2} and Definition
\ref{2Spin}, it suffices to show that the space $Z$ given by
(\ref{2Spinform}) is not $[2]$-1-complemented.

We need some preliminary observations concerning $AH_n$ and $BH_n$
which will lead to a formal relationship between $Z$ and the space
$G$ given by (\ref{5G}), with $(n-1)$ instead of $n$. Using the
notation (\ref{5Om-}), we have
$$
AH_n\,=\,{\rm Span}\{\omega_jP_n,\ \omega_{-j}P_n\, :\,1\leq j\leq
n\}
$$
For any $j\in \{-n,\ldots,-1\}\cup\{1,\ldots, n-1\}$ we let
$w'_j=i\omega_n\omega_j$. Then $\omega_j=i\omega'_j\omega_n$ and
for any $j=1,\ldots, n-1$, we have $\omega_j\omega_{-j}=
\omega'_j\omega'_{-j}$. Applying Lemma \ref{5Pn}, this yields
\begin{equation}\label{5Pn1}
P_n=\frac{1}{2}\,\bigl(1+i^{n-1}\omega'_{-n}\omega'_{n-1}\omega'_{-(n-1)}\cdots
\omega'_{1}\omega'_{-1}\bigr).
\end{equation}
Let $W\colon B(\Lambda_n)\to B(\Lambda_n)$ be the left
multiplication by $i\omega_n$. Later on we will use the obvious
fact that
\begin{equation}\label{5IsoW}
W\colon S^p(\Lambda_n)\longrightarrow S^p(\Lambda_n)\qquad\hbox{is
a complete isometry}.
\end{equation}
According to the above expression of $P_n$, the action of $W$ on
$AH_n$ is given by
\begin{equation}\label{5W1}
W(\omega_jP_n)=\omega'_j\Bigl(\frac{1+i^{n-1}\omega'_{-n}\omega'_{n-1}
\cdots\omega'_{1}\omega'_{-1}}{2}\Bigr)
\end{equation}
if $j$ belongs to $\{-n,\ldots,-1\}\cup\{1,\ldots, n-1\}$, and
\begin{equation}\label{5W2}
W(\omega_nP_n)=
i\,\Bigl(\frac{1+i^{n-1}\omega'_{-n}\omega'_{n-1}\cdots
\omega'_{1}\omega'_{-1}}{2}\Bigr).
\end{equation}
It is easy to check that the $(2n-1)$-tuple
$(\omega'_{-n},\omega'_{n-1},\omega'_{-(n-1)}, \ldots ,
\omega'_{1},\omega'_{-1})$ is a spin system. Moreover the product
of these spins is
$$
\omega'_{-n}\omega'_{n-1} \omega'_{-(n-1)}  \cdots \omega'_{1}
\omega'_{-1}\, =\, i\omega_n \omega_{-n}\omega_{n-1}
\omega_{-(n-1)}  \cdots \omega_{1}\omega_{-1},
$$
which is not a multiple of $1$. Thus by Lemma \ref{5UniqueSpin}
(2), there is a faithful $*$-representation $\pi\colon\C_{2n-1}\to
B(\Lambda_n)$ such that
$$
\pi(\omega_1)=\omega'_{-n},\ \pi(\omega_2)=\omega'_{n-1},\
\ldots,\, \pi(\omega_{2n-2})
=\omega'_1,\,\pi(\omega_{2n-1})=\omega'_{-1}.
$$
Furthermore, $tr(\pi(\omega_A))=0$ for any
$A\in\P_{2n-1}\setminus\{\emptyset\}$. Thus $tr(\pi(x))=2^nTr(x)$
for any $x\in\C_{2n-1}$ and we deduce (by Lemma \ref{5Ext}) that
\begin{equation}\label{5Iso}
2^{-\frac{n}{p}}\pi\colon L^p(\C_{2n-1})\longrightarrow
S^p(\Lambda_n) \quad\hbox{is a complete isometry.}
\end{equation}
Since $\pi$ is multiplicative, we see that $\pi(\rho_{n-1})=P_n$
by comparing (\ref{5rho0}) and (\ref{5Pn1}). Consequently, we have
$$
\pi(\rho_{n-1}E_{2n-1})=W(AH_n).
$$
Also we have
$$
\pi\bigl((1-\rho_{n-1})E_{2n-1}\bigr)=W(BH_n)
$$
and a thorough look at (\ref{5W1}) and (\ref{5W2}) actually shows
that
$$
\pi^{-1}W\kappa W^{-1}\pi\colon \rho_{n-1}E_{2n-1}\longrightarrow
(1-\rho_{n-1})E_{2n-1}
$$
is nothing but the mapping which takes $\rho_{n-1}y$ to
$(1-\rho_{n-1})y$ for any $y\in E_{2n-1}$.

As before we let $\pi_0,\pi_1\colon \C_{2n-2}\to \C_{2n-1}$ be the
left multiplications by $\rho_{n-1}$ and $(1-\rho_{n-1})$,
respectively. Then Lemma \ref{5tau1}, with $(n-1)$ instead of $n$,
yields the relation
\begin{equation}\label{5Relation}
\tau=(\pi_1^{-1}\pi^{-1}W)\circ\kappa\circ(W^{-1}\pi\pi_0) \colon
F_{n-1}\longrightarrow F_{n-1}.
\end{equation}
Set $\Gamma_i=W^{-1}\pi\pi_i  \colon \C_{2n-2}\to B(\Lambda_n)$,
for $i=1,2$. It follows from (\ref{5Iso0}), (\ref{5IsoW}) and
(\ref{5Iso}) that
\begin{equation}\label{5Transfer}
2^{-\frac{(n-1)}{p}}\, \Gamma_0,\ 2^{-\frac{(n-1)}{p}}\,\Gamma_1
\colon L^p(\C_{2n-2})\longrightarrow S^p(\Lambda_n) \quad\hbox{are
complete isometries}.
\end{equation}
Let us now assume that $Z$ is $[2]$-1-complemented in
$S^p(\Lambda_n \mathop{\otimes}\limits^2 H_1)
\mathop{\oplus}\limits^p S^p(\Lambda_n \mathop{\otimes}\limits^2
H_2)$. Then by (\ref{5Transfer}), the space
$$
\bigl\{(\Gamma_0^{-1}(x)\otimes a_1, \Gamma_1^{-1}\kappa(x)\otimes
a_2)\, :\, x\in AH_n^p\bigr\}\subset ( L^p(\C_{2n-2})
\mathop{\otimes}\limits^p S^p(H_1) )\mathop{\oplus}\limits^p
(L^p(\C_{2n-2}) \mathop{\otimes}\limits^p S^p(H_2))
$$
is $[2]$-1-complemented as well. According to (\ref{5Relation}),
this space coincides with the space $G$ of (\ref{5G}). By
Proposition \ref{5G1}, we obtain a contradiction.

The proof for the spinorial spaces of odd dimension is similar,
using Remark \ref{5G2} (2) in the place of Proposition \ref{5G1}.
We skip the details.
\end{proof}

\begin{remark}\label{5Final}
It follows from Proposition \ref{5G1}, its subsequent remark and
the proof of Theorem \ref{5Spin} that any spinorial space is
$1$-complemented. Also it follows from Remark \ref{5tau2} (2) and
the proof of Theorem \ref{5Spin} that the exchange map $\kappa$
defined by (\ref{2Kappa}) is an isometry on $AH^p_n$ for any $p$.
\end{remark}

\medskip
\section{Main results}
In this section we state our main results and prove Theorem
\ref{1Main}, mostly by combining results proved in the last three
sections. Throughout we assume that $1\leq p\not=2<\infty$.

\begin{theorem}\label{6Main}
Let $\H,\K$ be Hilbert spaces and let $X\subset S^p(\H,\K)$ be an
indecomposable subspace. The following are equivalent.
\begin{itemize}
\item [(i)] $X$ is completely $1$-complemented in $S^p(\H,\K)$.
\item [(ii)] $X$ is $[2]$-$1$-complemented in $S^p(\H,\K)$. \item
[(iii)] There exist index sets $I,J$ and an operator $a\in S^p(H)$
such that
$$
X\,\sim\, S^p_{I,J}\otimes a.
$$
\end{itemize}
\end{theorem}

\begin{proof}
The implication `$(i)\Rightarrow (ii)$' is obvious, and `$(iii)
\Rightarrow (i)$' follows from Lemma \ref{2Equiv2}. To prove the
hard implication `$(ii) \Rightarrow (iii)$', assume that $X$ is
$[2]$-$1$-complemented in $S^p(\H,\K)$ and that ${\rm dim}(X)>1$.
By Proposition \ref{3Sym} and Theorem \ref{5Spin}, $X$ is neither
a spinorial space of dimension $\geq 5$ nor a space of symmetric
or anti-symmetric matrices. Hence according to the Arazy-Friedman
Theorem \ref{2AF1}, $X$ is either a space of rectangular matrices
or is equivalent to a finite dimensional Hilbert space of the form
(\ref{2Hilbertform}). In the latter case, Proposition
\ref{4Hilbert} ensures that $X$ is actually equivalent to a space
of rectangular matrices. Then Proposition \ref{3Rect} finally
shows that $X$ satisfies (iii).
\end{proof}

\begin{proposition}\label{6Converse}
Let $(I_\alpha)_\alpha$ and $(J_\alpha)_\alpha$ be two families of
indices, and let
$$
u\colon \mathop{\oplus}\limits^p_{\alpha}
S^{p}_{I_\alpha,J_\alpha}\longrightarrow S^p(\H,\K)
$$
be a complete isometry. Then the range of $u$ is completely
$1$-complemented.
\end{proposition}

\begin{proof}
We may assume that $\H=\K$. First consider the case when the
family is a singleton, that is, we have index sets $I,J$ and a
complete isometry $u\colon S^p_{I,J}\to S^p(\H)$, and we wish to
show that its range is completely $1$-complemented. In the `square
case', that is,  $I=J$, this is a special case of \cite[Prop.
3.3]{JSR}. In fact it can also be quickly deduced from \cite[Th.
2.1]{AF1}. More generally, it is not hard to deduce the result
from the latter reference if $I\geq 2$ and $J\geq 2$. The sequel
of the proof is necessary only to treat the case when $I$ or $J$
is equal to $1$, although we will write it for general $I,J$. We
will show how to reduce to the `square case'.

We may consider a complete isometry $v\colon S^p_{J,I}\to
S^p(\H)$. For example the mappping $v$ defined by $v(x)={ }^t [u({
}^t x)]$ for any $x\in S^p_{J,I}$ is a complete isometry (here `${
}^t$' stands for the transposition). Recall that
$$
S^p_{I\times J}\simeq S^p_{I,J}\mathop{\otimes}\limits^p
S^p_{J,I}\simeq S^p_{J,I}\mathop{\otimes}\limits^p S^p_{I,J}.
$$
Hence the tensor map $u\otimes v$ extends to a complete isometry
$$
u\overline{\otimes} v\colon S^p_{I\times J}\longrightarrow S^p(\H)
\mathop{\otimes}\limits^p S^p(\H)\simeq
S^p(\H\mathop{\otimes}\limits^2\H).
$$
We know from the above discussion that the range of
$u\overline{\otimes} v$ is completely $1$-complemented. Thus there
exists a completely contractive mapping
$$
w\colon  S^p(\H) \mathop{\otimes}\limits^p S^p(\H)\longrightarrow
S^p_{I,J}\mathop{\otimes}\limits^p S^p_{J,I}
$$
such that $Q\circ u\overline{\otimes} v$ is the identity of
$S^p_{I,J}\mathop{\otimes}\limits^p S^p_{J,I}$. Let $z\in
S^p_{J,I}$ and $z^{*}\in (S^p_{J,I})^{*}$ such that $\langle
z^{*},z\rangle = \norm{z}=\norm{z^{*}}=1$. By e.g. \cite[Cor.
2.2.3]{ER}, $z^*$ is a complete contraction on $S^p_{J,I}$ hence
$Id\otimes z^{*}\colon S^p_{I,J}\otimes S^p_{J,I} \to S^p_{I,J}\,$
extends to a complete contraction
$$
Id\overline{\otimes}z^{*}\colon S^p_{I,J}\mathop{\otimes}\limits^p
S^p_{J,I} \longrightarrow S^p_{I, J}.
$$
Let $\widetilde{w}\colon S^p(\H)\to S^p_{I,J}$ be defined by
$$
\widetilde{w}(y)=\bigl[(Id\overline{\otimes}z^{*})\circ
w\bigr](y\otimes v(z)), \qquad y\in S^p(\H).
$$
This is a completely contractive map and for any $x\in S^p_{I,J}$,
we have
$$
\widetilde{w}\circ u(x) = \bigl[(Id\overline{\otimes}z^{*})\circ
w\bigr](u(x)\otimes v(z))=Id\otimes z^{*}(x\otimes z)=x.
$$
Thus $\widetilde{w}\circ u$ is the identity of $S^p_{I,J}$, which
proves the result.

\smallskip
We now consider the general case and we will apply results on
orthogonality reviewed in Section 2. Let $u$ be as in the
proposition and for any $\alpha$, let
$X_\alpha=u(S^p_{I_\alpha,J_\alpha})\subset S^p(\H,\K)$. Since $u$
is an isometry, it follows from (\ref{2Indec}) that the
$X_\alpha$'s are pairwise orthogonal. Thus there exist pairwise
orthogonal closed subspaces $H_\alpha\subset \H$ as well as
pairwise orthogonal closed subspaces $K_\alpha\subset \H$ such
that $X_\alpha\subset S^p(H_\alpha, K_\alpha)$. It follows from
the first part of this proof that for any $\alpha$, there is a
completely contractive projection
$$
P_\alpha\colon S^p(H_\alpha, K_\alpha)\longrightarrow
S^p(H_\alpha, K_\alpha)
$$
whose range equals $X_\alpha$. Moreover the $p$-direct sum
$\mathop{\oplus}\limits_\alpha^p S^p(H_\alpha, K_\alpha)\subset
S^p(\H,\K)$ is the range of a completely contractive projection
$Q\colon S^p(\H,\K)\to S^p(\H,\K)$. We can now define a mapping
$P\colon S^p(\H,\K)\to S^p(\H,\K)$ by letting
$$
P(z)=\bigl(P_\alpha(Q(z))\bigr)_\alpha,\qquad z\in S^p(\H,\K).
$$
Clearly $P$ is a completely contractive projection whose range is
equal to the range of $u$.
\end{proof}

\medskip
\begin{proof}[{\it Proof of Theorem \ref{1Main}.}]
The implication `$(iii) \Rightarrow (iv)$'  follows from Lemma
\ref{2Equiv2}, `$(iv) \Rightarrow (i)$'  is given by Proposition
\ref{6Converse} and `$(i) \Rightarrow (ii)$' is obvious. Now
assume (ii). By \cite[Prop. 2.2]{AF3}, $X$ can be written as the
$p$-direct sum $X=\mathop{\oplus}\limits^p_\alpha X_\alpha$ of
pairwise orthogonal indecomposable subspaces. Then it is plain
that each $X_\alpha$ is $[2]$-$1$-complemented as well. Applying
Theorem \ref{6Main} and an obvious direct sum argument, we deduce
that (iii) holds true.
\end{proof}

It follows from Theorem \ref{1Main} that if $X\subset S^p(\H,\K)$
and $Y\subset S^p(\H',\K')$ are completely isometric, then $X$ is
completely $1$-complemented if and only if $Y$ is completely
$1$-complemented.

\begin{remark} Let $X\subset B(\H,\K)$ be a $w^*$-closed subspace.
Using Theorem \ref{1Main} for $p=1$ and an elementary duality
argument, we find that if $X$ is the range of a $w^*$-continuous
completely contractive projection $B(\H,\K)\to B(\H,\K)$, then
there exist two families of Hilbert spaces $(H_\alpha)_\alpha$ and
$(K_\alpha)_\alpha$ such that $X$ is completely isometrically and
$w^*$-homeomorphically  isomorphic to
$\mathop{\oplus}\limits^{\infty}_\alpha B(H_\alpha,K_\alpha)$.

The converse does not hold true. Indeed there is an example in
\cite[Section 3]{ER2} of a $w^*$-continuous complete isometry
$u\colon B(H)\to B(\H)$ whose range cannot be the range of a
$w^*$-continuous completely contractive projection $B(\H)\to
B(\H)$.

For the sake of completeness, we note the following related result
going back to \cite{T}: A von Neumann algebra $M\subset B(\H)$ is
the range of a $w^*$-continuous contractive projection if and only
if it can be written as
$$
M\,\simeq\,\mathop{\oplus}\limits^{\infty}_\alpha B(H_\alpha),
$$
where `$\simeq$' indicates a von Neumann algebra identification.
It turns out that the same result holds true without the word
`contractive', see \cite[Remark 4.7]{LM}.
\end{remark}

\medskip
\section{Transpose map on the spin factor}
We recall that for any integer $n\geq 1$, the transpose map
$\tau\colon F_n\to  F_n$ is the linear isometry defined by
$$
\tau(1)=1,\quad\tau(\omega_j)=\omega_j\quad\hbox{for any }\,
j=1,\ldots, 2n,\quad\hbox{and}\quad
\tau(\omega_1\cdots\omega_{2n})=- \omega_1\cdots\omega_{2n}.
$$
In this section we consider the question whether $\tau$ is a
complete contraction (equivalently, a complete isometry) on
$F_n^p$ for $1\leq p<\infty$, or on $F_n=F_n^\infty\subset
\C_{2n}$, and we give applications.

In the sequel we use the notation
\begin{equation}\label{7Not}
s_0=1, \quad s_j=\omega_j\quad\hbox{for any }\, j=1,\ldots,
2n,\quad\hbox{and}\quad s_{2n+1}=\omega_1\cdots\omega_{2n}.
\end{equation}
For any $(2n+2)$-tuple of signs
$\Theta=(\theta_0,\theta_1,\ldots,\theta_{2n+1})\in\{-1,1\}^{2n+2}$,
one can more generally consider the map $\tau_{\Theta}\colon
F_n\to F_n$ defined by
$$
\tau_{\Theta} (s_j)=\theta_j s_j \quad\hbox{for any }\,
j=0,\ldots, 2n+1.
$$
The completely bounded norm of this map only depends on the parity
of the number of minus signs in the sequence $\Theta$. Indeed let
$\pi\colon \C_{2n}\to \C_{2n}$ be the $*$-isomorphism taking
$\omega_j$ to $\theta_0\theta_j\omega_j$ for any $j=1,\ldots, 2n$
and recall that $\pi\colon L^p(\C_{2n})\to L^p(\C_{2n})$ is a
complete isometry for any $1\leq p\leq \infty$. Then
$\theta_0\tau_{\Theta}$ is equal to the restriction of $\pi$ to
$F_n$ if $\theta_0\theta_1\cdots\theta_{2n+1}=1$ and is equal to
$\tau\circ\pi_{\vert F_n}$ if $\theta_0\theta_1\cdots
\theta_{2n+1}=-1$. Hence for any $1\leq p\leq \infty$,
$\tau_\Theta\colon F_n^p\to F_n^p$ is a complete isometry if
$\theta_0\theta_1\cdots\theta_{2n+1}=1$ whereas
$\cbnorm{\tau_\Theta\colon F_n^p\to F_n^p} = \cbnorm{\tau\colon
F_n^p\to F_n^p}$ if $\theta_0\theta_1\cdots\theta_{2n+1}=-1$.

We start with a precise estimate in the case $p=\infty$. Later on
we will find the same estimate for $p=1$.

\begin{proposition}\label{7Tau1}
For any $n\geq 1$, we have
$$
\bigcbnorm{\tau\colon F_n\longrightarrow F_n}\,=\,\frac{n+1}{n}\,.
$$
\end{proposition}

\begin{proof}
For any $j=0,1,\ldots,2n+1$, we let $\pi_j\colon \C_{2n}\to
\C_{2n}$ be the $*$-representation defined by letting
$\pi_j(x)=s_j^{*} xs_j$ for any $x\in\C_{2n}$. Of course, $\pi_0$
is just the identity map. It is easy to check that for any set
$A\in\P_{2n}$ and for any $1\leq j\leq 2n$, we have
\begin{align*}
\pi_j(\omega_A)=\omega_j\omega_A\omega_j \,=\, & \omega_A\qquad\,\
\hbox{if}\ \vert A\vert\ \hbox{ is even and }\ j\notin A;
\\ \,=\, &  - \omega_A\quad\ \hbox{if}\ \vert A\vert\ \hbox{ is even
and }\ j\in A;
\\ \,=\, &  - \omega_A\quad\ \hbox{if}\ \vert A\vert\ \hbox{ is odd
and }\ j\notin A;
\\ \,=\, &   \omega_A\qquad\,\ \hbox{if}\ \vert A\vert\ \hbox{ is odd
and }\ j\in A.\
\end{align*}
Then we have
$$
\pi_{2n+1}(\omega_A)=\omega_{2n}\cdots\omega_1\omega_A\omega_1
\cdots\omega_{2n} = (-1)^{\vert A\vert} \omega_A.
$$
It follows from these computations that for the $(2n+2)$-tuple
$\Theta = (-1,1,\ldots,1)$ we have
$$
2n\tau_{\Theta} =\pi_0 -\sum_{j=1}^{2n+1}\pi_j\,.
$$
Hence according to the discussion above this proposition, we have
$$
\bigcbnorm{\tau\colon F_n\longrightarrow
F_n}\,=\,\frac{1}{2n}\,\Bigcbnorm{ \pi_0
-\sum_{j=1}^{2n+1}\pi_j\colon F_n\longrightarrow F_n}.
$$
This yields the above estimate $\cbnorm{\tau}\leq (n+1)/n$.

\smallskip We now turn to the lower estimate. By the definition of
$\tau$, we have
$$
\Bignorm{\sum_{j=0}^{2n} s_j\otimes s_j \, +s_{2n+1}^{*}\otimes
s_{2n+1}}_{\footnotesize{\C}_{2n}\minten
\footnotesize{\C}_{2n}}\,\leq\,\cbnorm{\tau}\,
\Bignorm{\sum_{j=0}^{2n} s_j\otimes s_j \,  - s_{2n+1}^{*}\otimes
s_{2n+1}}_{\footnotesize{\C}_{2n}\minten \footnotesize{\C}_{2n}},
$$
where $\minten$ stands for the minimal (or spatial) tensor product
of $C^*$-algebras. Since $\C_{2n}\simeq M_{2^n}$ is a matrix
space, the bilinear map $\C_{2n}\times \C_{2n}\to
B(L^{2}(\C_{2n}))$ taking any $(a,b)$ to the mapping $T\to aTb$
(for $a,b\in\C_{2n}$ and $T\in L^2(\C_{2n})$) extends to an
isometric isomorphism
\begin{equation}\label{7Tau4}
\C_{2n}\minten \C_{2n}\,\simeq\,B\bigl(L^{2}(\C_{2n})\bigr).
\end{equation}
In this identification, $s_j^{*}\otimes s_j$ corresponds to
$\pi_j$ for any $j=0,\ldots, 2n+1$. Furthermore, it follows from
the first part of this proof that each $\pi_j$ is a diagonal
operator with respect to the orthonormal basis
$(\omega_A)_{A\in\footnotesize{\P_{2n}}}$, whose eigenvalues are
either $+1$ or $-1$. Moreover if $A\in\P_{2n}$ is such that
$\pi_j(\omega_A)=\omega_A$ for any $1\leq j\leq 2n$, then
$A=\emptyset$. We deduce that the eigenvalues of the diagonal
operator $\pi_0+\cdots+\pi_{2n}-\pi_{2n+1}$ are integers belonging
to $[-2n,2n]$. Thus
$$
\Bignorm{\sum_{j=0}^{2n} s_j\otimes s_j \,  - s_{2n+1}^{*}\otimes
s_{2n+1}}_{\footnotesize{\C}_{2n}\minten
\footnotesize{\C}_{2n}}=\Bignorm{\sum_{j=0}^{2n}\pi_j\,
-\pi_{2n+1}}_{L^{2}\to L^{2}}\leq 2n.
$$
On the other hand, $\pi_j(\omega_\emptyset)=1$ for any
$j=0,\ldots, 2n+1$, hence
$$
\Bignorm{\sum_{j=0}^{2n} s_j\otimes s_j \,  + s_{2n+1}^{*}\otimes
s_{2n+1}}_{\footnotesize{\C}_{2n}\minten
\footnotesize{\C}_{2n}}=\Bignorm{\sum_{j=0}^{2n}\pi_j\,
+\pi_{2n+1}}_{L^{2}\to L^{2}}= 2n+2.
$$
Consequently, we have $\cbnorm{\tau}\geq (n+1)/n$.
\end{proof}

\begin{remark}\label{7Stinespring} For a linear map $u\colon
\C_{2n}\to \C_{2n}$, the Wittstock factorization theorem asserts
that
$$
\cbnorm{u}=\inf\Bigl\{\bignorm{\sum_j a_j^* a_j}^{\frac{1}{2}}\,
\bignorm{\sum_j b_j^* b_j}^{\frac{1}{2}}\Bigr\},
$$
where the infimum runs over all finite families $(a_j)_j$ and
$(b_j)_j$ in $\C_{2n}$ such that
\begin{equation}\label{7Wittstock}
u(x)=\sum_j a_j^* x b_j,\qquad x\in\C_{2n}.
\end{equation}
(See e.g. \cite[Sect. 5.3]{ER}.) The above proof yields an
extension $u\colon \C_{2n}\to \C_{2n}$ of $\tau\colon F_n\to F_n$,
as well as a factorization of the type (\ref{7Wittstock}) such
that
$$\norm{\sum_j a_j^* a_j}=\norm{\sum_j b_j^* b_j}=\cbnorm{\tau}.$$
Indeed this is obtained by taking
$$
a_j=(2n)^{-\frac{1}{2}}\, s_{2n+1}s_j\quad\hbox{for }\, j=1,\ldots
2n+1,\qquad a_{2n+2}= (2n)^{-\frac{1}{2}}\, s_{2n+2},
$$
and then $b_j=a_j$ for $j=1,\ldots, 2n+1$ and
$b_{2n+2}=-a_{2n+2}$.
\end{remark}

As an application of the fact that $\tau\colon F_n\to F_n$ is not
completely contractive, we will now discuss the operator space
structures induced by  triple monomorphisms on Cartan factors of
type 4. See the last part of Section 2 for a brief account on this
class. We recall the well-known fact that for any $N\geq 1$,
$E_{N}$ is a Cartan factor of type 4. Moreover it follows from the
discussion in Section 5 that for any $n\geq 1$, the linear maps
$u_\pm\colon E_{2n+1}\to\C_{2n}$ defined by $u_\pm(w_{2n+1})= \pm
i^n\omega_1\cdots\omega_{2n}$, $u_\pm(1)=1$ and
$u_\pm(\omega_j)=\omega_j$ for $j=1,\ldots, 2n$ are triple
monomorphisms. Thus $F_n$ is a Cartan factor of type 4 and $\tau$
is a triple isomorphism.

For any integer $k\geq 1$ and any $z=[z_{ij}]\in M_k\otimes F_n$,
with $z_{ij}\in F_n$, we let
$$
\norm{z}_{M_k(F_n^\tau)}\, =\,\bignorm{[\tau(z_{ij})]}_{M_k(F_n)}.
$$
These matrix norms define an operator space structure on $F_n$,
denoted by $F_n^\tau$. Then we let $F_n\cap F_n^{\tau}$ be the
`intersection' of these two operator spaces defined by letting
$$
\norm{z}_{M_k(F_n\cap F_n^{\tau})}\,=\, \max\bigl\{
\norm{z}_{M_k(F_n)}\, ,\ \norm{z}_{M_k(F_n^{\tau})}\bigr\},\qquad
k\geq 1,\ z\in M_k\otimes F_n.
$$
(See \cite[Sect. 2.7 and 2.10]{P2}.)

\begin{proposition}\label{7Structures} Let $n\geq 1$ be an
integer.

\begin{itemize}
\item [(1)] Let $\H$ be a Hilbert space and let $u\colon F_n\to
B(\H)$ be a triple monomorphism. Then one of the following three
properties holds and they mutually exclude each other. Either
$u\colon F_n\to B(\H)$ is a complete isometry; or $u\colon
F_n^{\tau}\to B(\H)$ is a complete isometry; or $u\colon F_n\cap
F_n^{\tau}\to B(\H)$ is a complete isometry.

\item [(2)] Let $X$ be a Cartan factor of type 4, with ${\rm
dim}(X)=2n+2$. Then $X$ is completely isometric either to $F_{n}$
or to $E_{2n+1}$. Furthermore, $F_{n}$ and $E_{2n+1}$ are not
completely isometric.

\item [(3)] Let $\H$ be a Hilbert space. Then any triple
monomorphism $u\colon E_{2n}\to B(\H)$ is a complete isometry.
Consequently if $X$ is a Cartan factor of type 4, with ${\rm
dim}(X)=2n+1$, then $X$ is completely isometric to $E_{2n}$.
\end{itemize}
\end{proposition}

\begin{proof}
(1):\ Let $u\colon F_n\to B(\H)$ be a  triple monomorphism. We use
the description of such mappings established in \cite{AF2} and
given in \cite[p. 21]{AF3} in terms of the so-called irreducible
faithful representations. According to this description, and the
relationship between $AH_{n-1}$ and $F_n$ discussed in Section 5,
there exist Hilbert spaces $H_1, H_2$, two partial isometries
$a_1\in B(H_1), a_2\in B(H)$, and two partial isometries
$$
U\, ,V\colon (\Lambda_{2n}\mathop{\otimes}\limits^2 H_1)
\mathop{\oplus}\limits^2(\Lambda_{2n}\mathop{\otimes}\limits^2
H_2) \longrightarrow \H
$$
such that
$$
u(x) = V\bigl(x\otimes a_1, \tau(x)\otimes a_2\bigr) U^*
\quad\hbox{and}\quad  V^*V\bigl(x\otimes a_1, \tau(x)\otimes
a_2\bigr)U^*U= \bigl(x\otimes a_1, \tau(x)\otimes a_2\bigr)
$$
for any $x\in F_n$. This readily implies that for any $k\geq 1$
and any $z\in M_k\otimes F_n$, we have
\begin{align*}
\bignorm{(I_{M_k}\otimes u)z}_{M_k(B(\footnotesize{\H}))}\, & =\,
\max\{\norm{a_1}\norm{z}_{M_k(F_n)}\, ,\
\norm{a_2}\norm{(I_{M_k}\otimes \tau)z}_{M_k(F_n)}\}\\
& =\, \max\{\norm{a_1}\norm{z}_{M_k(F_n)}\, ,\
\norm{a_2}\norm{z}_{M_k(F_n^{\tau})}\}.
\end{align*}
Note that $\norm{a_i}\in\{0,1\}$. If $\norm{a_1}=1$ and
$\norm{a_2}=0$, then $u$ is a complete isometry on $F_n$. If
$\norm{a_1}=0$ and $\norm{a_2}=1$, then $u$ is a complete isometry
on $F_n^{\tau}$. Finally if $\norm{a_1}=\norm{a_2}=1$, then $u$ is
a complete isometry on $F_n\cap F_n^{\tau}$.

The fact that these three cases mutually exclude each other simply
means $\tau$ is not a complete isometry, which was shown in
Proposition \ref{7Tau1}.

\smallskip
(2):\ We first observe that $E_{2n+1}$ is completely isometric to
$F_n\cap F_n^{\tau}$. This follows from our discussion in Section
5. Indeed if $\pi\colon \C_{2n+1}\to
\C_{2n}\mathop{\oplus}\limits^{\infty}\C_{2n}$ is the
$*$-isomorphism given by (\ref{5Decomp2}), then
$\pi(E_{2n+1})=\{(x,\tau(x))\, :\, x\in F_n\}$. Since $\pi$ is a
complete isometry, the result follows at once.

Now let $X$ be a Cartan factor of type 4, with ${\rm
dim}(X)=2n+2$. It follows from (1) that $X$ is completely
isometric either to $F_n$, to $F_n^{\tau}$, or to $F_n\cap
F_n^{\tau}$. Since $F_n$ and $F_n^{\tau}$ are completely isometric
(via $\tau$), we deduce using the above observation that $X$ is
actually completely isometric to either $F_n$ or $E_{2n+1}$.

It remains to prove that $E_{2n+1}$ is not completely isometric to
$F_n$. We have noticed above that there is a (natural) completely
isometric triple isomorphism $J\colon E_{2n+1}\to F_n\cap
F_n^{\tau}$. Let $v\colon E_{2n+1}\to E_{2n+1}$ be an arbitrary
linear isometry and recall that this forces $v$ to be a triple
isomorphism. Applying part (1) of this proposition to $u=vJ^{-1}$,
we obtain that for any $k\geq 1$ and for any $x\in M_k\otimes
E_{2n+1}$, $\norm{(I_{M_k}\otimes v)x}_{M_k(E_{2n+1})}$ is equal
either to $\norm{x}_{M_k(E_{2n+1})}$, or to
$\norm{J(x)}_{M_k(F_n)}$, or to $\norm{J(x)}_{M_k(F_n^{\tau})}$.
In any case, we have
$$
\norm{(I_{M_k}\otimes u)x}_{M_k(E_{2n+1})} \leq
\norm{x}_{M_k(E_{2n+1})}
$$
Applying the same reasoning  to $v^{-1}$ we obtain that $v$ is
actually a complete isometry. Thus any isometry of $E_{2n+1}$ is a
complete one. Since $F_n$ admits an isometry which is not a
complete one (namely, $\tau$), these two spaces cannot be
completely isometric.

\smallskip (3) As in (1), this follows from the description
of triple monomorphism $u\colon E_{2n}\to B(\H)$ given by \cite[p.
21]{AF3} (and \cite{AF2}). Indeed there exist a Hilbert space $H$,
a partial isometry $a\in B(H)$, and two partial isometries
$$
U\, ,V\colon (\Lambda_{2n}\mathop{\otimes}\limits^2 H)
\longrightarrow \H
$$
such that
$$
u(x) = V(x\otimes a)U^* \quad\hbox{and}\quad V^*V(x\otimes a)
U^*U= x\otimes a
$$
for any $x\in F_n$. This factorization readily implies that $u$ is
a complete isometry.
\end{proof}

\begin{remark}\label{7Identif}
Using the description of triple monomorphisms on Cartan factors of
type 1-3 established in \cite{AF2} and given in \cite[p. 21]{AF3},
one obtains analogs of the above proposition for these factors, as
follows. See also \cite{NR1}.

\smallskip (1)
Let $H,\H$ be Hilbert spaces. Any triple monomorphism $\S(H)\to
B(\H)$ is a complete isometry. If ${\dim H}\geq 5$, any triple
monomorphism $\A(H)\to B(\H)$ is a complete isometry.

\smallskip (2)
Let $n,m\geq 2$ be integers. Then for any triple monomorphism
$u\colon M_{n,m}\to B(\H)$, one of the following three properties
holds and they mutually exclude each other. Either $u\colon
M_{n,m}\to B(\H)$ is a complete isometry; or $u\colon M_{n,m}^{\rm
op}\to B(\H)$ is a complete isometry; or $u\colon M_{n,m}\cap
M_{n,m}^{\rm op}\to B(\H)$ is a complete isometry. Thus if $X$ is
a Cartan factor which is triple equivalent to $M_{n,m}$, then $X$
is completely isometric either to $M_{n,m}$ or to $M_{n,m}^{\rm
op}$ or to $M_{n,m}\cap M_{n,m}^{\rm op}$. Further if $n\not=m$,
the latter three spaces are pairwise non completely isometric to
each other. Lastly for $n=m$, we note that $M_n$ and $M_n^{\rm
op}$ are completely isometric (via the transposition map), whereas
$M_n$ and $M_n\cap M_n^{\rm op}$ are not completely isometric.

Let $H,K$ be Hilbert spaces. The above results extend to the case
of triple monomorphisms $u\colon B(H,K)\to B(\H)$, provided that
$H$ or $K$ is finite dimensional. However the classification of
all operator space structures induced by triple monomorphisms
$B(H,K)\to B(\H)$ when $H$ and $K$ are infinite dimensional is
unclear.

For a partial description of all possible operator space
structures induced by triple monomorphisms $B(\Cdb,K)\to B(\H)$,
see \cite{NRR}.
\end{remark}

\bigskip
We now turn to the study of $\tau\colon F_n^p\to F_n^p$ for finite
$p$. For any integer $N\geq 1$, let $\Ddb_N$ be the finite set $
\{-1,1\}^{N}$ equipped with its uniform probability measure
$\Pdb$, and consider the Rademacher functions
$\varepsilon_1,\ldots,\varepsilon_N\colon \Ddb_N\to\{-1,1\}$
defined by letting $\varepsilon_j(\Theta)=\theta_j$ for any
$\Theta=(\theta_1,\ldots,\theta_N)$ in $\Ddb_N$ and any $1\leq
j\leq N$. We will need the following lemma, in which $N=2n$ and
$\norm{\ }_p$ stands for the norm in $L^p(\Ddb_{2n})$.

\begin{lemma}\label{7Rad1} For any complex numbers
$\alpha_0,\alpha_1,\ldots,\alpha_{2n+1}$, we have
$$
\bigcbnorm{\tau\colon F_n^p\longrightarrow F_n^p}\,\geq\,\frac{\,
\Bignorm{\alpha_0 + \sum_{j=1}^{2n}\alpha_j\varepsilon_j
+\alpha_{2n+1}\prod_{j=1}^{2n}\varepsilon_j}_p\,}{\,
\Bignorm{\alpha_0 + \sum_{j=1}^{2n}\alpha_j\varepsilon_j
-\alpha_{2n+1}\prod_{j=1}^{2n}\varepsilon_j}_p\,}.
$$
\end{lemma}

\begin{proof} This is a continuation of the proof of Proposition
\ref{7Tau1}. Let $\alpha_0,\alpha_1,\ldots,\alpha_{2n+1}$ be
complex numbers. We will show that
\begin{equation}\label{7Rad2}
\Bignorm{\sum_{j=0}^{2n+1} \alpha_j s_j^{*}\otimes
s_j}_{L^p(\footnotesize{\C}_{2n})\mathop{\otimes}\limits^p
L^p(\footnotesize{\C}_{2n})}\,=\,\Bignorm{\alpha_0 +
\sum_{j=1}^{2n}\alpha_j\varepsilon_j
+\alpha_{2n+1}\prod_{j=1}^{2n}\varepsilon_j}_p.
\end{equation}
Changing $\alpha_{2n+1}$ into $-\alpha_{2n+1}$ and applying the
definition of $\tau$, this implies the result.

We first note that the identification (\ref{7Tau4}) induces an
isometric isomorphism
$$
L^p(\C_{2n})\mathop{\otimes}\limits^p L^p(\C_{2n})\,\simeq\,
S^p\bigl(L^2(\C_{2n})\bigr),
$$
which yields
\begin{equation}\label{7Rad3}
\Bignorm{\sum_{j=0}^{2n+1} \alpha_j s_j^{*}\otimes
s_j}_{L^p(\footnotesize{\C}_{2n})\mathop{\otimes}\limits^p
L^p(\footnotesize{\C}_{2n})}\, =
\Bigl(\frac{1}{2^{2n}}\Bigr)^{\frac{1}{p}}\,
\Bignorm{\sum_{j=0}^{2n+1}\alpha_j \pi_j
}_{S^p(L^2(\footnotesize{\C}_{2n}))}.
\end{equation}
The subspace of $S^p\bigl(L^2(\C_{2n})\bigr)$ of operators which
are diagonal with respect to the orthonormal basis
$(\omega_A)_{A\in\footnotesize{\P}_{2n}}$ is equal to
$\ell^p_{\footnotesize{\P}_{2n}}$. To  any $A\in\P_{2n}$, let us
associate the $2n$-tuple
$\Theta_A=(\theta_{1},\ldots,\theta_{2n})\in\Ddb_{2n}$ defined by
$\theta_j=1\,\Leftrightarrow j\notin A$. Then it follows from the
proof of Proposition \ref{7Tau1} that
$$
\pi_j(\omega_A)\,=\,\Bigl(\varepsilon_j\bigl( \prod_{i=1}^{2n}
\varepsilon_i\bigr)\Bigr)(\Theta_A)\,\omega_A,\qquad j=1,\ldots,
2n.
$$
Thus in the isometric isomorphism
$\ell^p_{\footnotesize{\P}_{2n}}\simeq L^p(\Ddb_{2n})$ induced by the
correspondance $A\leftrightarrow \Theta_A$, the diagonal operator
$\pi_j$ corresponds to $\varepsilon_j\bigl(\prod_{i=1}^{2n}
\varepsilon_i\bigr)$ for any $1\leq j\leq 2n$. Likewise,
$\pi_{2n+1}$ corresponds to $\prod_{i=1}^{2n} \varepsilon_i$.
Hence
\begin{equation}\label{7Rad4}
\Bigl(\frac{1}{2^{2n}}\Bigr)^{\frac{1}{p}}\,
\Bignorm{\sum_{j=0}^{2n+1}\alpha_j \pi_j
}_{S^p(L^2(\footnotesize{\C}_{2n}))}  \,=\, \Bignorm{\alpha_0 +
\sum_{j=1}^{2n}\alpha_j\varepsilon_j\bigl(\prod_{i=1}^{2n}
\varepsilon_i\bigr)
+\alpha_{2n+1}\prod_{i=1}^{2n}\varepsilon_i}_p.
\end{equation}
Now set $\eta_j=\varepsilon_{j}\bigl(\prod_{i=1}^{2n}
\varepsilon_i\bigr)$ for any $1\leq j\leq 2n$. Then we have
$$
\prod_{j=1}^{2n} \eta_j =\prod_{j=1}^{2n}
\varepsilon_j\qquad\hbox{and}\qquad \varepsilon_j=\eta_j\bigl(
\prod_{i=1}^{2n} \eta_i\bigr),\quad\ 1\leq j\leq 2n.
$$
Consider $(\theta_1,\ldots,\theta_{2n})\in\{-1,1\}^{2n}$ and let
$\theta=\theta_1\theta_2\cdots\theta_{2n}$ be the product of these
$\pm 1$. It follows from above that
\begin{align*}
\Pdb\bigl(\{\eta_j=\theta_j\ \forall\, j=1,\ldots , 2n\}\bigr) &
\,=\, \Pdb\bigl(\{\varepsilon_j=\theta_j\theta\
\forall\, j=1,\ldots , 2n\}\bigr)\\
& \,=\, \Pdb\bigl(\{\varepsilon_j=\theta_j\ \forall\, j=1,\ldots ,
2n\}\bigr)\,=\,\frac{1}{2^{2n}}\,.
\end{align*}
Thus $(\eta_1,\ldots,\eta_{2n})$ has the same distribution as
$(\varepsilon_1,\ldots,\varepsilon_{2n})$, and hence
$$
\Bignorm{\alpha_0 +
\sum_{j=1}^{2n}\alpha_j\varepsilon_j\bigl(\prod_{i=1}^{2n}
\varepsilon_i\bigr)
+\alpha_{2n+1}\prod_{i=1}^{2n}\varepsilon_i}_p\,=\,
\Bignorm{\alpha_0 + \sum_{j=1}^{2n}\alpha_j\varepsilon_j
+\alpha_{2n+1}\prod_{i=1}^{2n}\varepsilon_i}_p.
$$
Together with (\ref{7Rad3}) and (\ref{7Rad4}), this implies the
equality (\ref{7Rad2}).
\end{proof}

\begin{proposition}\label{7TauS1}
For any $n\geq 1$, we have
$$
\bigcbnorm{\tau\colon F_n^1\longrightarrow
F_n^1}\,=\,\frac{n+1}{n}\,.
$$
\end{proposition}

\begin{proof} The upper estimate clearly follows from the proof of
Proposition \ref{7Tau1}. For the lower estimate we consider
$$
f=1 + \sum_{j=1}^{2n}\varepsilon_j
-(-1)^n\prod_{j=1}^{2n}\varepsilon_j\qquad\hbox{and}\qquad g=1 +
\sum_{j=1}^{2n}\varepsilon_j +(-1)^n \prod_{j=1}^{2n}\varepsilon_j
$$
in $L^1(\Ddb_{2n})$. According to Lemma \ref{7Rad1}, it suffices
to show that
$$
\frac{\norm{f}_1}{\norm{g}_1}\,=\,\frac{n+1}{n}\, .
$$
For any $i=1,\ldots, 2n$, let $\rho_i\colon L^1(\Ddb_{2n})\to
L^1(\Ddb_{2n})$ be induced by the $*$-representation which takes
$\varepsilon_i$ to $-\varepsilon_i$ and which takes
$\varepsilon_j$ to $\varepsilon_j$ for any $j\not=i$. Then
$\rho_i$ is an isometry and
$$
\rho_i(\prod_{j=1}^{2n}\varepsilon_j)=-
\prod_{j=1}^{2n}\varepsilon_j.
$$
Then we let $\rho=-\rho_1\circ\cdots\circ\rho_{2n}$. A few
elementary computations (left to the reader) yield
$$
2ng\,=\, f\, +\,\sum_{i=1}^{2n}\rho_i(f)\, +\rho(f).
$$
Let $k$ be the $\Zdb$-valued function on $\Ddb_{2n}$ defined as
$$
k={\rm Card}\bigl\{j\in\{1,\ldots, 2n\}\, :\,
\varepsilon_j=1\bigr\}
$$
and let $m=k-n$. Then
$$
f=1+k -(2n-k) - (-1)^n(-1)^{2n-k}\,=\, 1 -(-1)^m + 2m,
$$
hence $f$ is valued in $4\Zdb$. Moreover for any $i=1,\ldots, 2n$,
we have
$$
f-\rho_i(f) =2\bigl(\varepsilon_i
-\prod_{j=1}^{2n}\varepsilon_j\bigr),
$$
hence $f-\rho_i(f)$ is valued in $\{-4,0,4\}$. Consequently, $f$
and $\rho_i(f)$ have the same sign everywhere on $\Ddb_{2n}$.
Likewise $f$ and $\rho(f)$ have the same sign. This implies that
$$
\bignorm{f\, +\,\sum_{i=1}^{2n}\rho_i(f)\, +\rho(f)}_1\, =\,
\norm{f}_1+\,\sum_{i=1}^{2n}\norm{\rho_i(f)}_1\,
+\norm{\rho(f)}_1.
$$
We deduce that $2n\norm{g}_1= (2n+2)\norm{f}_1$, which concludes
the proof.
\end{proof}

\begin{theorem}\label{7Main} Let $n\geq 1$ be an integer and
let $1\leq p<\infty$. The following are equivalent.
\begin{itemize}
\item [(i)] $\cbnorm{\tau\colon F_n^p\to F_n^p}\leq 1$
(equivalently, $\tau\colon F_n^p\to F_n^p$ is a complete
isometry). \item [(ii)] $p$ is an even integer and $2n\geq p$.
\end{itemize}
\end{theorem}

\begin{proof}
(ii)$\,\Rightarrow\,$ (i): Assume that $p=2q$, $q$ is an integer
and $q\leq n$. Again we use the notation (\ref{7Not}). To prove
(i), we consider $a_0,a_1,\ldots, a_{2n+1}$ in $S^p$ and aim at
showing that
\begin{equation}\label{7Main1}
\Bignorm{\sum_{j=0}^{2n} a_j\otimes s_j\, + a_{2n+1}\otimes
s_{2n+1}}_{S^p[L^p(\footnotesize{\C}_{2n})]}^p\,=\,
\Bignorm{\sum_{j=0}^{2n} a_j\otimes s_j\, - a_{2n+1}\otimes
s_{2n+1}}_{S^p[L^p(\footnotesize{\C}_{2n})]}^p.
\end{equation}
We have
\begin{align*}
\Bigl\vert\sum_{j=0}^{2n+1} a_j\otimes s_j\Bigr\vert^{2q} & \,=\,
\Bigl(\sum_{j=0}^{2n+1} a_j\otimes
s_j\Bigr)^{*}\Bigl(\sum_{j=0}^{2n+1} a_j\otimes
s_j\Bigr)\,\cdots\,\Bigl(\sum_{j=0}^{2n+1} a_j\otimes
s_j\Bigr)^{*}\Bigl(\sum_{j=0}^{2n+1} a_j\otimes s_j\Bigr) \\ &
\,=\,\sum_{0\leq j_1,\ldots, j_{2q}\leq 2n+1}
a_{j_1}^{*}a_{j_2}\ldots a_{j_{2q-1}}^{*} a_{j_{2q}}\otimes
s_{j_1}^{*}s_{j_2}\ldots s_{j_{2q-1}}^{*} s_{j_{2q}}.
\end{align*}
Here the sum runs over all
$(j_1,\ldots,j_{2q})\in\{0,\ldots,2n+1\}^{2q}$. Recall that $Tr$
and $tr$ denote the canonical traces on $\C_{2n}$ and
$B(\ell^{2})$ respectively. By the above calculation, the left
hand side of (\ref{7Main1}) is equal to
$$
\sum_{0\leq j_1,\ldots, j_{2q}\leq 2n+1}
tr\bigl(a_{j_1}^{*}a_{j_2}\ldots a_{j_{2q-1}}^{*}
a_{j_{2q}}\bigr)\,Tr\bigl( s_{j_1}^{*}s_{j_2}\ldots
s_{j_{2q-1}}^{*} s_{j_{2q}}\bigr).
$$
Changing $a_{2n+1}$ into $- a_{2n+1}$, we see that the right hand
side of (\ref{7Main1}) is equal to
$$
\sum_{0\leq j_1,\ldots, j_{2q}\leq 2n+1} (-1)^{\vert \{k\, :\,
j_k=2n+1\}\vert}\, tr\bigl(a_{j_1}^{*}a_{j_2}\ldots
a_{j_{2q-1}}^{*} a_{j_{2q}}\bigr)\,Tr\bigl(
s_{j_1}^{*}s_{j_2}\ldots s_{j_{2q-1}}^{*} s_{j_{2q}}\bigr).
$$
To show the equality (\ref{7Main1}) it therefore suffices to check
that if a $2q$-tuple $(j_1,\ldots,j_{2q})$ in
$\{0,\ldots,2n+1\}^{2q}$ is such that the cardinal $\vert \{k\,
:\, j_k=2n+1\}\vert$ is an odd number, then
$$
Tr\bigl( s_{j_1}^{*}s_{j_2}\ldots s_{j_{2q-1}}^{*}
s_{j_{2q}}\bigr)=0.
$$
Suppose that $\vert \{k\, :\, j_k=2n+1\}\vert= 2m+1$, for some
integer $m\geq 0$. Recall that $s_j^{*}=s_j$ for any $j\leq 2n$,
that $s_{2n+1}^*= (-1)^n s_{2n+1}$ and that for any $0\leq
j,j'\leq 2n+1$, the operators $s_j$ and $s_{j'}$ either commute or
anticommute. Note also that $s_{2n+1}^{2}=(-1)^n$. Thus we have
$$
\bigl\vert Tr\bigl( s_{j_1}^{*}s_{j_2}\ldots s_{j_{2q-1}}^{*}
s_{j_{2q}}\bigr)\bigr\vert \,=\, \bigl\vert Tr\bigl( s_{j_1}
s_{j_2}\ldots s_{j_{2q-1}} s_{j_{2q}}\bigr)\bigr\vert \,=\,
\bigl\vert Tr\bigl(s_{2n+1}^{2m+1}S\bigr)\bigr\vert\,=\,
\bigl\vert Tr\bigl(s_{2n+1}S\bigr)\bigr\vert,
$$
where $S\in\C_{2n}$ is the product of $(2q-(2m+1))$ operators
belonging to the set $\{1,\omega_1,\ldots,\omega_{2n}\}$. Since
$q\leq n$, this product has at most $(2n-1)$ factors. Since
$\omega_j^{2}=1$ for any $j\leq 2n$, we deduce that there exists
an integer $r\geq 1$ and $r$ distinct integers $i_1,\ldots i_r$
between $1$ and $2n$ such that $s_{2n+1}S=
\omega_{i_1}\cdots\omega_{i_r}$. Then the trace of $s_{2n+1}S$ is
zero.

\smallskip
(i)$\,\Rightarrow\,$ (ii): Let us assume that $p=2q$, $q$ is an
integer, $q>n$, and let us show that $\tau\colon F_n^p\to F_n^p$
is not completely contractive. We set
$$
P(t)\,=\,\Bignorm{1+\sum_{j=1}^{2n}\varepsilon_j\, +
t\prod_{j=1}^{2n}\varepsilon_j}_{p}^p,\qquad t\in\Rdb.
$$
Then $P$ is a polynomial and according to Lemma \ref{7Rad1}, the
fact that $\tau\colon F_n^p\to F_n^p$ is not completely
contractive is equivalent to $P$ not be even. Let $c_1=P'(0)$ be
the coefficient of degree 1. Set $\varepsilon_0=1$ for
convenience. Then by a computation similar to the one in the first
part of this proof, one obtains that
$$
c_1\,=\,\sum_{(j_1,\ldots,j_{2q})\in\Gamma}
\Edb\Bigl(\prod_{j=1}^{2n}\varepsilon_j \,\prod_{k: j_k\not=
2n+1}\varepsilon_{j_k}\Bigr),
$$
where $\Gamma\subset\{0,\ldots,2n+1\}^{2q}$ is the set of all
$2q$-tuples $(j_1,\ldots,j_{2q})$ for which there is a unique
$1\leq k\leq 2q$ such that $j_k=2n+1$. Here $\Edb$ denotes the
conditional expectation on $(\Ddb_{2n},\Pdb)$. Observe that for
any $(j_1,\ldots,j_{2q})\in\Gamma$,
$$
\Edb\Bigl(\prod_{j=1}^{2n}\varepsilon_j \,\prod_{k: j_k\not=
2n+1}\varepsilon_{j_k}\Bigr)\, =\, 0\ \hbox{or}\ 1.
$$
Moreover for the $2q$-tuple defined by letting $j_k=k$ for any
$1\leq k\leq 2n+1$ and $j_k=0$ for any $k\geq 2n+2$, then the
above conditional expectation is equal to $1$. (We use that $q>n$
to define this particular $2q$-tuple.) We deduce that $c_1>0$, and
hence that $P$ is not even.

\smallskip
Let us now assume that $\gamma=\frac{p}{2}$ is not an integer and
let us show that $\tau$ is not a complete contraction on $F_n^p$.
For any positive real number $a>0$ and any $t\in\Rdb$, we set
$$
\Phi(a,t)\,=\,\Bignorm{1+ia^{\frac{1}{2}}\Bigl(1+\sum_{j=1}^{2n}
\varepsilon_j\,+t\prod_{j=1}^{2n}\varepsilon_j\Bigr)}_p^p.
$$
Using Lemma \ref{7Rad1} again, it suffices to show that for some
positive real number $a>0$, the function $\Phi(a,\cdotp)$ is not
even. We have
$$
\Phi(a,t)\,=\,\Edb\biggl[\Bigl(1+a\Bigl(1+\sum_{j=1}^{2n}
\varepsilon_j\,+t\prod_{j=1}^{2n}\varepsilon_j\Bigr)^{2}\Bigr)^\gamma\biggr],
$$
hence $\Phi$ extends to a $C^\infty$ function on a neighborhood of
zero.

Suppose that $\Phi(a,\cdotp)$ is even for any $a>0$. Then
$\frac{\partial^{n+1}\Phi}{\partial a^{n+1}}(0,\cdotp)$ also is an
even function. On a neighborhood of zero, we have
$$
\frac{\partial^{n+1}\Phi}{\partial a^{n+1}}(a,t)\, =\,
\delta_{\gamma,n}\,\Edb\biggl[\Bigl(1+\sum_{j=1}^{2n}
\varepsilon_j\,+t\prod_{j=1}^{2n}\varepsilon_j\Bigr)^{2(n+1)}
\Bigl(1+a\Bigl(1+\sum_{j=1}^{2n}
\varepsilon_j\,+t\prod_{j=1}^{2n}\varepsilon_j\Bigr)^{2}\Bigr)^{\gamma
-(n+1)}\biggr],
$$
where $\delta_{\gamma,n}=\gamma(\gamma-1)\cdots(\gamma-n)$. Since
$\gamma $ is not an integer, this constant is non zero and we
obtain that for any $t\in\Rdb$,
$$
\delta_{\gamma,n}^{-1}\,\frac{\partial^{n+1}\Phi}{\partial
a^{n+1}}(0,t)\, =\, \Edb\biggl[\Bigl(1+\sum_{j=1}^{2n}
\varepsilon_j\,+t\prod_{j=1}^{2n}\varepsilon_j\Bigr)^{2(n+1)}\biggr].
$$
When we showed above that $\tau\colon F_n^{2(n+1)}\to
F_n^{2(n+1)}$ is not a complete contraction, we showed that the
above function of $t$ is not even. Hence we obtain a
contradiction.
\end{proof}

We now give an application to an extension problem.

\begin{corollary}\label{7Extension} Let $p\geq 4$ be an even
integer. There exist a subspace $X\subset S^p$ and a completely
bounded map $u\colon X\to S^p$ which has no bounded extension
$S^p\to S^p$.
\end{corollary}

We will need the following classical averaging argument. Let $Z$
be a reflexive Banach space, let $G$ be an amenable group and let
$\lambda\colon G\to B(Z)$ be a strongly continuous representation
such that $\lambda(g)\colon Z\to Z$ is a (necessarily onto)
isometry for any $g\in G$. We say that a subspace $X\subset Z$ is
invariant if $\lambda(g)$ maps $X$ into $X$ for any $g\in G$ and
we say that a bounded linear map $u\colon X\to Z$ is a multiplier
if
$$
\lambda(g)\bigl(u(x)\bigr)\,=\, u(\lambda(g)x),\qquad x\in X.
$$

\begin{lemma}\label{7Mult} Let $v\colon Z\to Z$ be a bounded linear map,
and assume that $X\subset Z$ is invariant and that $v_{\vert
X}\colon X\to Z$ is a multiplier. Then there exist a multiplier
$w\colon Z\to Z$ such that $\norm{w}\leq\norm{v}$ and $v_{\vert
X}= w_{\vert X}$.
\end{lemma}

\begin{proof}
Let $\psi\in L^{\infty}(G)^*$ be a translation invariant mean on
$G$. For any $x\in Z$ and $y\in Z^*$, consider the function
$$
f_{x,y}(g)=\bigl\langle
\lambda(g^{-1})\bigl(v(\lambda(g)x)\bigr),y\bigr\rangle,\qquad
g\in G.
$$
Then we may define $w\in B(Z)$ by letting
$$
\langle w(x),y\rangle=\psi(f_{x,y})
$$
and it is clear that the operator $w$ satisfies the required
conditions.
\end{proof}

\smallskip
\begin{proof} [{\it Proof of Corollary \ref{7Extension}.}]
Let $p\geq 4$ be an even integer and let $n=p/2$. We let $p'$ be
the conjugate number of $p$ and for $q\in\{p,p'\}$, we let
$\tau_q\colon F_n^q\to F_n^q$ denote the transposition. According
to Theorem \ref{7Main}, we have $\cbnorm{\tau_p}=1$ whereas
$c=\cbnorm{\tau_{p'}}>1$.

We fix an integer $k\geq 1$ and we consider
$$
Z_{k,q}= L^q(\C_{2n})\mathop{\otimes}\limits^q \cdots
\mathop{\otimes}\limits^q L^q(\C_{2n}) \mathop{\otimes}\limits^q
S^q\qquad\hbox{and}\qquad X_{k,q}= F_n^q \mathop{\otimes}\limits^q
\cdots \mathop{\otimes}\limits^q F_n^q \mathop{\otimes}\limits^q
S^q.
$$
We will exhibit a complete contraction $u\colon X_{k,p}\to
Z_{k,p}$ such that $\norm{v}\geq c^k$ for any bounded linear map
$v\colon Z_{k,p}\to Z_{k,p}$ extending $u$. Since $Z_{k,p}$ is
completely isometric to $S^p$, the result follows at once using a
standard direct sum argument.

For any $\Theta=(\theta_1,\ldots,\theta_{2n})\in \Ddb_{2n}$, we
let $\pi_\Theta\colon L^q( \C_{2n})\to L^q(\C_{2n})$ denote the
$L^q$-version of the $*$-representation $\C_{2n}\to \C_{2n}$
taking $\omega_j$ to $\theta_j\omega_j$ for any $j=1,\ldots, 2n$.
This is an isometry and $\Theta\mapsto \pi_\Theta$ is a
representation of $\Ddb_{2n}$ on $L^q(\C_{2n})$. Then let $\Tdb$
be the unit circle and consider elements of $S^q$ as infinite
matrices $[t_{rs}]_{r,s\geq 1}$ in the usual way. For any
$\alpha=(\alpha_r)_{r\geq 1}$ and $\beta=(\beta_s)_{s\geq 1}$ in
$\Tdb^{\infty}$, let $\gamma(\alpha,\beta)\colon S^q\to S^q$ be
the linear mapping taking any matrix $[t_{rs}]\in S^q$ to
$[\alpha_r t_{rs}\beta_s]$. Then $\gamma\colon\Tdb^{\infty}\times
\Tdb^{\infty}\to B(S^q)$ is a strongly continuous isometric
representation. We will apply Lemma \ref{7Mult} to the group
$$
G=\underbrace{\Ddb_{2n}\times\cdots\times \Ddb_{2n}}_{k\, {\rm
times}}\times\Tdb^{\infty}\times \Tdb^{\infty}
$$
and to the representation $\lambda\colon G\to B(Z_{k,q})$ defined
by letting
$$
\lambda(\Theta^1,\ldots,\Theta^{k},\alpha,\beta) =
\pi_{\Theta^1}\otimes\cdots\otimes\pi_{\Theta^k}\otimes
\gamma(\alpha,\beta).
$$
Indeed, $G$ is amenable and it follows from the above discussion
that $\lambda$ is a strongly continuous isometric representation.
Let $E_{rs}$ denote the canonical matrix units in $S^q$. It is
easy to check (left to the reader) that a bounded linear map
$w\colon Z_{k,q}\to Z_{k,q}$ is a multiplier if and only if $w$ is
diagonal with respect to the elements $\omega_{A_1}\otimes \cdots
\otimes\omega_{A_k}\otimes E_{rs}$, for $A_1,\ldots A_k\in\P_{2n}$
and $r,s\geq 1$.

Clearly the space $X_{k,p}$ is invariant. Let
$$
u=\tau_p\otimes\cdots\otimes\tau_p\otimes I_{S^p}\colon
X_{k,p}\longrightarrow Z_{k,p}.
$$
Then $u$ is a multiplier and $\cbnorm{u}\leq 1$. Let $v\colon
Z_{k,q}\to Z_{k,q}$ be a bounded extension of $u$. By Lemma
\ref{7Mult} there exist a multiplier $w\colon Z_{k,q}\to Z_{k,q}$
extending $u$ and such that $\norm{w}\leq\norm{v}$. Now consider
the adjoint map $w^*\colon Z_{k,p'}\to Z_{k,p'}$. Let
$$
\F=\bigl\{\emptyset,\{1\},\ldots,\{2n\},\{1,\ldots,2n\}\bigr\}
$$
and recall that $F_n$ is the linear span of $\{\omega_A\, :\,
A\in\F\}$. For any $A_1,\ldots, A_k\in\F$  and any $r,s\geq 1$, we
have
\begin{align*}
\bigl\langle w^* (\omega_{A_1}\otimes \cdots &
\otimes\omega_{A_k}\otimes E_{rs}), \,\omega_{A_1}\otimes \cdots
\otimes\omega_{A_k}\otimes E_{rs}\bigr\rangle \\ &= \bigl\langle
\omega_{A_1}\otimes \cdots \otimes\omega_{A_k}\otimes E_{rs},\,
\tau(\omega_{A_1})\otimes \cdots \otimes\tau(\omega_{A_k})\otimes
E_{rs}\bigr\rangle \\ & =(-1)^m,
\end{align*}
where $m$ is the number of $j$'s such that $A_j=\{1,\ldots,2n\}$.
Since $w$ is a multiplier, $w^*$ is a multiplier as well, hence we
deduce from above that
$$
w^* (\omega_{A_1}\otimes \cdots \otimes\omega_{A_k}\otimes
E_{rs})=(-1)^m \omega_{A_1}\otimes \cdots
\otimes\omega_{A_k}\otimes E_{rs}.
$$
This shows that
$$
w^*_{\vert
X_{k,p'}}\,=\,\tau_{p'}\otimes\cdots\otimes\tau_{p'}\otimes
I_{S^{p'}}.
$$
Consequently we have
$$
c^k=\cbnorm{\tau_{p'}\otimes\cdots\otimes\tau_{p'}} =
\norm{\tau_{p'}\otimes\cdots\otimes\tau_{p'}\otimes
I_{S^{p'}}}\leq\norm{w^*}=\norm{w}\leq \norm{v},
$$
and this concludes the proof.
\end{proof}

We mention  that instead of Theorem \ref{7Main}, one can use some
results from \cite{Neu} to prove the above corollary.

\bigskip
As a complement we will prove that Corollary \ref{7Extension}
extends to the case $p=1$. We start with a simple consequence of
the noncommutative Khintchine inequalities. We refer e.g. to
\cite[Sect. 9.8]{P2} for these inequalities. In the sequel we let
$$
\Phi_N={\rm Span} \{\omega_1,\ldots,\omega_N\}\subset\C_N
$$
and we let $\Phi_N^1$ be that space regarded as a subspace of
$L^1(\C_N)$. Also we let ${\rm Rad}_q^N={\rm Span}
\{\varepsilon_1,\ldots,\varepsilon_N\}\subset L^q(\Ddb_N)$ for
$q\in\{1,\infty\}$.

\begin{lemma}\label{7K} There is a constant $C\geq 1$ such that
for any $N\geq 1$ and for any $a_1,\ldots,a_N$ in $S^1$, we have
$$
C^{-1}\Bignorm{\sum_{j=1}^{N} a_j\otimes
\omega_j}_{S^1\mathop{\otimes}\limits^1\Phi_N^1}\leq
\Bignorm{\sum_{j=1}^{N} a_j\otimes
\varepsilon_j}_{S^1\mathop{\otimes}\limits^1 {\rm Rad}_N^1}\leq C
\Bignorm{\sum_{j=1}^{N} a_j\otimes
\omega_j}_{S^1\mathop{\otimes}\limits^1\Phi_N^1}.
$$
\end{lemma}

\begin{proof} This result is a simple consequence of \cite[Th.
3.7]{JO}, which is more general. We give a specific proof of
independent interest. We will use classical notation and
techniques from operator space theory (see \cite{P2}). The symbol
$\approx$ will stand for a complete isomorphism whose isomorphism
constants do not depend on the dimension. With this notation, the
result to be proved is that
$$
\Phi_N^{1}\approx {\rm Rad}_1^{N}.
$$

The noncommutative Khintchine inequalities on $S^1$ mean that
$({\rm Rad}_1^{N})^{*}\approx R_N\cap C_N$. This implies that
\begin{equation}\label{7K1}
\bigl(L^{1}(\C_N)\mathop{\otimes}\limits^1 {\rm Rad}_N^1\bigr)^{*}
\, \approx\, \C_N\minten(R_N\cap C_N).
\end{equation}
For $q\in\{1,\infty\}$, we let
$$
\lambda\colon\Ddb_N\longrightarrow
B(L^q(\Ddb_N))\qquad\hbox{and}\qquad \mu
\colon\Ddb_N\longrightarrow B(L^q(\C_N))
$$
be the natural representations of $\Ddb_N$. Namely for any
$\Theta=(\theta_1,\ldots,\theta_N)\in\Ddb_N$, $\lambda(\Theta)$
(resp. $\mu(\Theta)$) is the $*$-representation taking
$\varepsilon_j$ to $\theta_j\varepsilon_j$ (resp. $\omega_j$ to
$\theta_j\omega_j$) for any $j$. These maps are complete
isometries.

Let $(e_1,\ldots,e_N)$ denote the canonical basis of $R_N\cap
C_N$, let $Z^{1}_N$ be the linear span of the
$\omega_j\otimes\varepsilon_j$  in
$L^{1}(\C_N)\mathop{\otimes}\limits^1{\rm Rad}_N^1$ and let
$Z^{\infty}_N$ be the linear span of the $\omega_j\otimes e_j$ in
$\C_N\minten(R_N\cap C_N)$. Let $P\colon
L^{1}(\C_N)\mathop{\otimes}\limits^1{\rm Rad}_N^1\to
L^{1}(\C_N)\mathop{\otimes}\limits^1{\rm Rad}_N^1$ be the
orthogonal projection onto $Z^{1}_N$. It is easy to check that
$$
P\,=\,\int_{\footnotesize{\Ddb}_N}\mu(\Theta)\otimes\lambda(\Theta)\,
d\Pdb(\Theta)\,,
$$
where $\Pdb$ be the uniform probability on $\Ddb_N$. Hence $P$ is
completely contractive. Passing to the adjoints and using
(\ref{7K1}), this implies
$$
Z^{1 *}_N\,\approx\, Z^\infty_N.
$$
For any $a_1,\ldots,a_N\in B(H)$, we have
$$
\sum_{j=1}^{N}(a_j\otimes\omega_j)^{*}(a_j\otimes\omega_j)\, =
\sum_{j=1}^{N} a_j^{*}a_j\otimes \omega_j^*\omega_j\,=\,
\Bigl(\sum_{j=1}^{N} a_j^{*}a_j\Bigr)\otimes 1,
$$
and similarly,
$\sum_{j}(a_j\otimes\omega_j)(a_j\otimes\omega_j)^{*}$ is equal to
$\bigl(\sum_{j} a_ja_j^{*}\bigr)\otimes 1$. This implies that
$$
\Bignorm{\sum_{j=1}^{N} a_j\otimes \omega_j\otimes
e_j}_{B(H)\minten\footnotesize{\C_N}\minten (R_N\cap C_N)}\, =\,
\Bignorm{\sum_{j=1}^{N} a_j\otimes e_j}_{B(H)\minten  (R_N\cap
C_N)}.
$$
Thus $Z^\infty_N$ is completely isometrically isomorphic to
$R_N\cap C_N$.

Now let $a_1,\ldots,a_N\in S^1$. We have $\bignorm{\sum_{j}
a_j\otimes \omega_j}_{S^1\mathop{\otimes}\limits^1\Phi_N^1}=
\bignorm{\sum_{j} a_j\otimes
\theta_j\omega_j}_{S^1\mathop{\otimes}\limits^1\Phi_N^1}$ for any
$\Theta\in\Ddb_N$. Taking the average, this implies that
$$
\Bignorm{\sum_{j=1}^{N} a_j\otimes
\omega_j}_{S^1\mathop{\otimes}\limits^1\Phi_N^1}\,=\,
\Bignorm{\sum_{j=1}^{N} a_j\otimes
\omega_j\otimes\varepsilon_j}_{S^1\mathop{\otimes}\limits^1\Phi_N^1
\mathop{\otimes}\limits^1{\rm Rad}_N^1}.
$$
Thus $Z^1_N$ is completely isometrically isomorphic to
$\Phi_N^{1}$. Consequently, $\Phi_N^{1 *}\approx R_N\cap C_N$, and
hence $\Phi_N^{1}\approx {\rm Rad}_N^{1}$.
\end{proof}

\begin{proposition}\label{7Ext1}
There exist a subspace $X\subset\ell^1$ and a completely bounded
map $u\colon X\to S^1$ which has no bounded extension $\ell^1\to
S^1$.
\end{proposition}

\begin{proof} Suppose that this statement  is false
and for any $N\geq 1$, let $u_N\colon {\rm Rad}_1^N\to L^1(\C_N)$
be the linear mapping taking $\varepsilon_j$ to $\omega_j$ for any
$j=1,\ldots,N$. Then there is a constant $K\geq 1$ (not depending
on $N$) such that $u_N$ has an extension $v_N\colon L^1(\Ddb_N)\to
L^1(\C_N)$ satisfying $\norm{v_N}\leq K\cbnorm{u_N}$. The argument
in Lemma \ref{7Mult} shows that $u_N$ actually has an extension
$w_N\colon L^1(\Ddb_N)\to L^1(\C_N)$ such that
$w_N\circ\lambda(\Theta)=\mu(\Theta)\circ w_N$ for any
$\Theta\in\Ddb_N$, and $\norm{w_N}\leq \norm{v_N}$. Arguing as in
Corollary \ref{7Extension} we find that the restriction of
$w_N^{*}$ to $\Phi_N$ coincides with the canonical mapping
$\Phi_N\to{\rm Rad}^{\infty}_N$ which takes $\omega_j$ to
$\varepsilon_j$ for any $j$. Since $\Phi_N=\ell^{2}_N$ and ${\rm
Rad}_{N}^{\infty}=\ell^{1}_N$ isometrically we find that
$$
\sqrt{N}=\bignorm{Id\colon \ell^{2}_N\longrightarrow
\ell^{1}_N}\leq \norm{w_N}\leq K\cbnorm{u_N}.
$$
However Lemma \ref{7K} ensures that $\sup_N\cbnorm{u_N}<\infty$,
which yields a contradiction.
\end{proof}

We proved in Proposition \ref{5OrthoEbis} that the orthogonal
projection onto $E_N$ is not completely contractive. For
completeness we shall now give an asymptotic estimate of its
completely bounded norm, based on Lemma \ref{7K}. For any two
sequences $(\alpha_N)_{N\geq 1}$ and $(\beta_N)_{N\geq 1}$ of
positive real numbers, we write $\alpha_N\asymp\beta_N$ to say
that there exists a constant $C\geq 1$ such that
$C^{-1}\alpha_{N}\leq\beta_N\leq C\alpha_{N}$ for any $N\geq 1$.

\begin{corollary}
Let $P_N\colon\C_N\to\C_N$ be the orthogonal projection onto
$E_N$. Then
$$
\cbnorm{P_N}\asymp\sqrt{N}.
$$
\end{corollary}

\begin{proof}
Let $W_N\colon \C_N\to \Phi_N \hookrightarrow \C_N$ be the
orthogonal projection onto $\Phi_N$. It is clear that
$\cbnorm{P_N}\asymp\cbnorm{W_N}$. The adjoint $W_N^{*}\colon
\Phi_N^*\to  L^1(\C_N)$ coincides the `identity mapping' from
$\Phi_N^*$ onto $\Phi_N^{1}$. By Lemma \ref{7K} and the
noncommutative Khintchine inequalities, we have
$\Phi_N^{1}\approx{\rm Rad}^1_N\approx (R_N\cap C_N)^{*}$, hence
$\cbnorm{W_N^{*}}\asymp\cbnorm{Id\colon \Phi_N^{*}\to(R_N\cap
C_N)^{*}}$. Thus by duality,
$$
\cbnorm{W_N} \asymp\cbnorm{Id\colon R_N\cap
C_N\longrightarrow\Phi_N}.
$$
It is shown in \cite[p. 221]{P2} that $\cbnorm{Id\colon R_N\cap
C_N\to\Phi_N}\asymp\sqrt{N}$, and this estimate completes the
proof.
\end{proof}

\medskip By an averaging argument it is easy to see that for any
projection $Q\colon\C_N\to\C_N$ whose range is equal to $E_N$, we
have $\cbnorm{P_N}\leq\cbnorm{Q}$. Thus the `completely bounded
projection constant' of $E_N$ is $\asymp\sqrt{N}$. This result
appeared in \cite[Cor. 3.12]{JO}.

\begin{remark}
We do not know which numbers $p\in [1,2)\cup(2,\infty)$ have the
property $(\E)$ that there exist a subspace $X\subset S^p$ and a
completely bounded map $X\to S^p$ without any bounded extension
$S^p\to S^p$. We just proved that $(\E)$ holds true for any $p$
belonging to the set $B=\{1\}\cup\{2k\, :\, k\geq 2\}$.

Let $d_{cb}$ denote the completely bounded Banach-Mazur distance
of operator spaces (see \cite[p. 20]{P2}). For any $m\geq 1$, the
function $(p,q)\mapsto d_{cb}(S^p_m,S^q_m)$ is continuous. A
thorough look at the proofs of Proposition \ref{7Ext1} and
Corollary \ref{7Extension} together with a simple continuity
argument based on the above fact therefore shows that any $p\in B$
admits a neighborhood $\V_p$ such that $(\E)$ holds true for any
$q\in\V_p$.
\end{remark}

\bigskip\noindent{\bf Acknowledgements.} We thank the referee for
bringing the paper \cite{JO} to our attention.

\end{document}